\documentclass[11pt]{amsart}
\usepackage[foot]{amsaddr}

\usepackage[space]{grffile} 
\usepackage{etoolbox} 
\makeatletter 
\patchcmd\Gread@eps{\@inputcheck#1 }{\@inputcheck"#1"\relax}{}{}
\makeatother

\usepackage[shortlabels]{enumitem}
\setlist[itemize]{label=\raisebox{2pt}{$\scriptscriptstyle{\bullet}$}}

\numberwithin{equation}{section}
\usepackage{verbatim}

\usepackage{appendix}
\usepackage{geometry}
\usepackage{amssymb}
\usepackage{amsmath}
\usepackage{amsthm}
\usepackage{mathtools}
\usepackage{mathrsfs}
\DeclareFontFamily{U}{rsfs}{\skewchar\font127 }
\DeclareFontShape{U}{rsfs}{m}{n}{%
   <-6> rsfs5
   <6-8> rsfs7
   <8-> rsfs10
}{}
\usepackage{bm}

\usepackage{cleveref}
\usepackage{framed}
\usepackage{graphicx}
\usepackage{xcolor}
\usepackage{listings}
\usepackage{multirow}
\usepackage{multicol}
\usepackage{indentfirst}
\usepackage{booktabs}

\usepackage{tikz}
\usetikzlibrary{shapes.geometric}
\usepackage{tikz-cd}
\usetikzlibrary{patterns, arrows.meta, calc}

\usepackage{float}
\usepackage{algorithm}
\usepackage{algorithmicx}
\usepackage{algpseudocode}
\usepackage{subfigure}
\usepackage{lineno}
\usepackage{setspace}

\theoremstyle{plain}
\newtheorem{theorem}{Theorem}[section]
\newtheorem{lemma}[theorem]{Lemma}

\newtheorem{proposition}[theorem]{Proposition}

\theoremstyle{definition}

\theoremstyle{remark}
\newtheorem{remark}{Remark}[section]

\newcommand*{\be}[1]{\begin{equation}\label{#1}}
\newcommand*{\ee}{\end{equation}}

\DeclareMathOperator{\grad}{grad}
\DeclareMathOperator{\hess}{hess}
\DeclareMathOperator{\airy}{airy}
\DeclareMathOperator{\curl}{curl}
\DeclareMathOperator{\sym}{sym}
\DeclareMathOperator{\divergence}{div}
\DeclareMathOperator{\tr}{tr}
\DeclareMathOperator{\tf}{tf}
\DeclareMathOperator{\rot}{rot}
\DeclareMathOperator{\deff}{def}
\DeclareMathOperator{\inc}{inc}

\renewcommand{\div}{\divergence}

\DeclareMathOperator{\id}{id}

\newcommand{\jump}[1]{[\![ #1 ]\!]}

\newcommand{\rmC}{\mathrm{C}}
\newcommand{\rmD}{\mathrm{D}}
\newcommand{\rmL}{\mathrm{L}}
\newcommand{\rmH}{\mathrm{H}}
\newcommand{\rmP}{\mathrm{P}}
\newcommand{\bbI}{\mathbb{I}}
\newcommand{\bbR}{\mathbb{R}}
\newcommand{\bbS}{\mathbb{S}}

\newcommand{\calE}{\mathcal{E}}
\newcommand{\calF}{\mathcal{F}}

\newcommand{\calK}{\mathcal{K}}

\newcommand{\calO}{\mathcal{O}}

\newcommand{\calS}{\mathcal{S}}
\newcommand{\calT}{\mathcal{T}}
\newcommand{\calV}{\mathcal{V}}

\newcommand{\bs}{{\scriptscriptstyle{\bullet}}}
\newcommand{\rmd}{\mathrm{d}}
\newcommand{\transp}{\mathrm{T}}

\newcommand{\urmH}{\underline{\rmH}}
\newcommand{\uurmH}{\underline{\underline{\rmH}}}

\newcommand{\uB}{\underline{B}}

\newcommand{\uuB}{\underline{\underline{B}}}
\newcommand{\huuB}{\hat{\uuB}}

\newcommand{\Hdiv}{\urmH_{\div}}
\newcommand{\Hdivdiv}{\uurmH_{\div\div}}
\newcommand{\Hsymdiv}{\uurmH_{\div}}
\newcommand{\Hrotrot}{\uurmH_{\rot\rot}}
\newcommand{\Hinc}{\uurmH_{\inc}}
\newcommand{\Htr}{\underline{\underline{\rmH}}_{\tr}}
\newcommand{\RM}{\mathrm{RM}}

\newcommand{\lag}{\mathrm{Lag}}
\newcommand{\reg}{\mathrm{Reg}}

\newcommand{\tang}{\tau}
\newcommand{\norm}{\nu}

\newcommand{\CdR}{\mathcal{H}_{\mathrm{dR}}}

\newcommand{\skipthis}[1]{{\color{red} [SKIPPED ELEMENT]}}

\newcommand{\shc}[1]{}

\title{Regge metrics with enhanced trace}

\author{Snorre H. Christiansen}
\address{Department of Mathematics, University of Oslo, PO Box 1053 Blindern, 0316 Oslo, Norway}
\email{snorrec@math.uio.no}

\author{Ting Lin}
\address{School of Mathematics, Peking University, Beijing 100871, P.R.China}
\email{lintingsms@pku.edu.cn}
\thanks{The work of Ting Lin was supported by NSFC Project No. 123B2014.}

\begin{document}

\begin{abstract}
We introduce variants of Regge finite element metrics with enhanced properties of the trace. In particular the trace operator is surjective to a finite element space of continuous functions. Multiplying these scalar functions by the identity tensor brings one back to the finite element space of metrics. The metrics can be based on high order polynomials and be constructed on refinements, such as the Clough-Tocher or Worsey-Farin splits. Potential applications to general relativity, incompressible elasticity and conformal geometry are sketched.
\end{abstract}

\maketitle


\bigskip


\bigskip


\section{Introduction}

\subsection{Overview}
Many partial differential equations can be interpreted in the context of Riemannian geometry, and the tools of differential geometry can then be applied. In particular relevant differential operators can be identified. A systematic development of discrete Riemannian geometry could similarly provide a good background for the discretization of geometric partial differential equations. This paper is concerned with a finite element point of view. Discrete spaces are defined in such a way that they are adapted to given differential operators, for instance in the sense of preserving exact sequences. 

The discretization of the de Rham complex of differential forms is by now relatively well understood. For spaces that can be traced back to, in particular, Whitney and Nédélec, see \cite{Hip02,ArnFalWin06}. Various smoother and possibly non-conforming spaces have also been developed (some references are provided below). They are important for the Stokes problem in fluid mechanics.  A major motivation for developing FEEC was the design of finite element methods for elasticity \cite[\S 8.8]{Arn18}. In \cite{Arn02} the discretization of the equations of general relativity (GR) is also mentioned as a long term goal. The BGG construction \cite{ArnHu21} is one of the main new tools.

How to incorporate metric tensors in finite element complexes is an ongoing research topic, to which this paper aims to contribute. Finite elements related to more general Lie groups can also be considered. For instance the Yang-Mills equations generalize the Maxwell equations through the choice of an additional Lie group \cite{ChrHal12JMP}. The methods described here relate to the orthogonal group and also the conformal group, which allows scalings. 

Regge calculus \cite{Reg61}, which is a discrete approach to GR, was cast in a finite element language in \cite{Chr04M3AS, Chr11NM, Chr24}. In \cite{Chr04M3AS} Regge calculus was described similarly to the Whitney forms, in \cite{Chr11NM} the distributional linearized curvature was studied and the space of metrics was included in a discrete elasticity complex, and in \cite{Chr24} the nonlinear curvature defined by Regge was rigorously related to the one known in differential geometry for smooth metrics.

High order Regge metrics were defined in \cite{Li18}. Curvatures of Regge metrics have been investigated in \cite{Gaw20,GopNeuSchWar23,GawNeu25}. Finite element spaces of metrics, inspired by Regge but with extra regularity guaranteeing square integrable curvature, were defined in \cite{ChrHu23}. Further developments in this direction can be found in \cite{CheHua22,ChrEtAl24}.

In many applications of Riemannian geometry the trace of one metric relative to another is an important quantity. A background metric, which can very well be Euclidean, might be given and the relative trace of the foreground metric might vary, for instance due to conformal scalings. It appears that Regge metrics do not behave well with respect to such scalings, for all applications. This can be interpreted similarly to locking phenomena that appear in discrete elasticity, in the incompressible limit \cite[\S 8]{BofBreFor13}. Indeed, the kernel of the linearized curvature operator on Regge elements consists of deformations of continuous $\rmP_1$ vector fields and imposing divergence freeness on the latter is problematic on general meshes. In other methods involving Regge calculus, natural scalings can be well behaved \cite{Luo04}. In that reference the scaling is defined to take place at vertices (continuous $\rmP_1$), not at faces where the trace of Regge metrics lives a priori (discontinuous $\rmP_0$).

In this paper we design variants of Regge metrics on a Euclidean background that behave well with respect to the trace operator. In addition to tangential-tangential continuity, we impose continuous trace. Given tangential-tangential continuity, this amounts to additional normal-normal continuity on faces. Furthermore the discrete complexes, instead of starting with Lagrange elements as indicated above, will start with a Stokes element with continuous pressure. We will use examples from \cite{ChrHu18,GuzLisNei22}. The curvatures will remain distributional. For applications in GR it seems relevant to not try to impose too much regularity, as the tangent space of a simplicial complex is ill defined at vertices and other low dimensional faces.

Many of the complexes we construct feature both finite elements in the sense of Ciarlet and distributions. To check that such complexes are good, it seems a minimum requirement that their cohomology is correct. On contractible domains they should be exact, except at index $0$, which can be deduced from commuting projections if they are available. Sometimes an argument based on the Euler-Poincaré characteristic is enough to conclude. In the presence of distributions cohomological correctness can be non-trivial. Compare with \cite[Theorem 2]{ChrHu23}. We have included results on cohomology that extend and build on those of \cite{HuLinZha25,ChrHuLin26}.

The rest of this introduction is devoted to reviewing notations, detailing some Sobolev differential complexes and, finally, to provide more specific information about our motivations. Our main results then follow. Section \ref{sec:dimtwo} is devoted to the two dimensional setting, whereas Section \ref{sec:dimthree} concerns dimension three. We give several examples of spaces of Regge metrics with enhanced trace, fit them into discrete complexes and check that these complexes have correct cohomology. Some of the results concerning cohomologies are relegated to the Appendix.

\subsection{Notations}
For generalities about the mathematical theory surrounding elasticity we refer to \cite{Cia25}. For finite elements we refer to \cite{Bra07,BofBreFor13,Arn18}. We summarize our main notations here. We focus on the case of dimension $d=3$, the case $d=2$ being similar.

Let $\Omega$ be a polyhedral Lipschitz domain in $\bbR^3$. We choose a triangulation $\mathcal{T}$ of $\overline \Omega$. We use $\calK, \calF$, $\calE$, $\calV$ to denote the sets of cells, faces, edges and vertices in the triangulation $\calT$, and $\calF^{\circ}, \calE^{\circ}$, $\calV^{\circ}$ for the sets of interior faces, edges and vertices.

When the vertices are indexed by a set $I$, as in $\calV = \{x_i \ : \ i \in I\}$, the barycentric coordinate corresponding to vertex $x_i$ is denoted $\lambda_i$. This notation is also applied locally. For a tetrahedron $T = [x_0, x_1, x_2, x_3]$, we let $\lambda_0, \lambda_1, \lambda_2, \lambda_3$ denote the corresponding barycentric coordinates. Similarly, we may introduce $\lambda_0$, $\lambda_1$ for an edge $e = [x_0, x_1]$, and $\lambda_0, \lambda_1,\lambda_2$ for a face $f = [x_0, x_1, x_2]$. 

For an oriented edge $e$, $\tang_e$ denotes the unit oriented tangent vector. For an oriented face $f$, $\norm_f$ denotes the unit oriented normal.

For a $3 \times 3$ matrix $\sigma$ and an edge $e$, $\sigma_{\tang \tang} = \tang_e^\transp\sigma \tang_e \in \bbR$ denotes its tangential-tangential component. Similarly $\sigma_{\norm \norm} = \norm_f^\transp \sigma \norm_f$ denotes its normal-normal component with respect to a face $f$. Moreover, $\Pi_f v = - (v \times \norm_f) \times \norm_f$ denotes the tangential part (orthogonal projection) of a vector $v \in \mathbb{R}^3$ on a face $f$. Given a basis of the tangent space of the face, tangent vectors are sometimes identified with an element in $\bbR^2$. For a $3 \times 3$ matrix $\sigma$, $\Pi_f \sigma \Pi_f$ is its tangential-tangential part (taking into account that $\Pi_f^\transp = \Pi_f$).

The space of symmetric $d\times d$ matrices is denoted $\bbS_d$. Given a face $f$, the space of symmetric matrices that are spanned by matrices of the form $u v^\transp + v u^\transp$, with vectors $u,v$ that are tangent to $f$, is denoted $\bbS (f)$. In principle the same notation could be applied for edges $e$, but is often resolved due to  $\bbS(e) =  \bbR \tang_e \tang_e^\transp \approx \bbR$. 
If $\sigma \in \bbS_3$ we have $\Pi_f^\transp \sigma \Pi_f \in \bbS(f)$. 
In the presence of a basis of $f$ it will correspondingly be identified with an element of $\bbS_2$. 

For a matrix, we use $\tr$ to denote its trace, in the sense of sum of its $d$ diagonal entries. It can be computed in any orthonormal basis of $\bbR^3$. It follows that, with respect to any face, $\tr \sigma = \sigma_{\norm \norm} + \tr \sigma_{\tang \tang}$. 
The operator providing the tracefree part of matrix is denoted $\tf$.

In dimension two we write $u^{\perp} = (-u_2,u_1)$ for any $u \in \bbR^2$. It just encodes a direct rotation by a right angle. It can be written also as multiplication by the matrix $J$, defined as:
\begin{equation}
 J =  \begin{bmatrix} 0 & -1\\
    1 & 0
  \end{bmatrix}.
\end{equation}

Spaces of infinitesimal rigid motions are denoted $\RM_d$ with subscript $d$ to denote the dimension of the ambient space. Most importantly:
\begin{align}
    \RM_2 & = \{\bbR ^2 \ni x \mapsto a x^{\perp} + b \ : \  a \in \bbR,\ b \in \bbR^2\},\\
    \RM_3 & = \{ \bbR^3 \ni x \mapsto a \times x + b \ : \ a,\ b \in \bbR^3\}.
\end{align}
We will also consider rigid motions on faces of triangulations, denoted with $\RM(f)$ for a face $f$.

Differential operators are understood in several ways. We consider that $\bbR^d$ consists of column vectors. In dimension $2$, the $\rot$ operator maps scalars to row vectors. It also maps row vectors to scalars. But it maps a column vector to a matrix, row wise. It also maps a matrix to a column vector, row wise. More generally, in any dimension standard differential operators such as $\grad$, $\curl$ and $\div$ are taken to act row wise on matrices.

If $\sigma$ is a $2\times2$ matrix we can apply $\rot$ once to get a column vector. When we apply $\rot$ once more it is understood that it is on the transpose of this column vector, so that we get a scalar. In other words, by $\rot\rot$ we actually mean $\rot\transp\rot$. This is to be consistent with the usual notation for $\div\div$, where $\div$ is applied row wise first and then on the resulting column vector. In other words it is actually $\div \transp \div$.

In two dimensions, we introduce the strain complexes and the $\div \div$ complexes, related by algebraic pointwise isomorphisms. 
\begin{equation}\label{eq:straindivdiv}
    \begin{tikzcd}
        0 \ar[r] &  \rmC^{\infty} \otimes \bbR^2 \ar[r,"\sym\grad"] \ar[d,"\perp"] & \rmC^{\infty} \otimes \bbS_2 \ar[r,"\rot\rot"] \ar[d,"\perp\perp"] & \rmC^{\infty} \ar[r] \ar[d]& 0 \\ 
        0 \ar[r] &  \rmC^{\infty} \otimes \bbR^2 \ar[r,"\sym\rot"] & \rmC^{\infty} \otimes \bbS_2 \ar[r,"\div\div"] & \rmC^{\infty} \ar[r] & 0 .
    \end{tikzcd}
\end{equation}
For matrix-valued functions $\sigma$ we have used:
\begin{equation}
\sigma^{\perp\perp} = J^\transp \sigma J = \begin{bmatrix}
    \sigma_{22} & -\sigma_{21} \\ -\sigma_{12} & \sigma_{11}
\end{bmatrix} = (\tr \sigma) \bbI - \sigma^\transp.
\end{equation}
Notice that symmetric matrices are mapped to symmetric matrices, by this operation. 

We require some standard operators from 3D elasticity. We denote by:
\begin{equation}
  \deff u = \frac{1}{2} (\grad u + \transp \grad u),
\end{equation}
the deformation or symmetric gradient of a vector $u$, by
\begin{equation}
  \inc \sigma = \curl \transp \curl \sigma,
\end{equation}
the incompatibility tensor of a 2-tensor $\sigma$, and by $\div \sigma$ the divergence of a 2-tensor $\sigma$.

Tensor valued Sobolev spaces are underlined to indicate the order of the tensors. In particular:
\begin{align}
  \urmH^s(\Omega) & = \rmH^s(\Omega) \otimes \bbR^d,\\
  \uurmH^s(\Omega) & = \rmH^s(\Omega) \otimes \bbR^{d\times d}.
\end{align}
When there is no exponent $s$, we mean $s = 0$.

We use $\CdR^{\bs}(\Omega)$ to denote the de Rham cohomology of $\Omega$ (functoriality properties and exact regularity of representatives are sometimes kept implicit).

\subsection{Continuous strain complexes with enhanced trace}

This subsection discusses the continuous enhanced strain complex in two dimensions. The paper also covers the three dimensional case, but we postpone the introduction of some of the corresponding machinery.

We first specify a space of symmetric matrix valued functions, as follows:
\begin{equation}
  \Hrotrot(\Omega) := \{\sigma \in \rmL^2(\Omega) \otimes \bbS_2 \ : \ \rot\rot \sigma \in \rmH^{-1}(\Omega) \}.
\end{equation}
A strain complex with relatively low regularity then reads:
\begin{equation}
\label{eq:complex-2d-sobolev}
\begin{tikzcd}
0 \ar[r] &  \urmH^1(\Omega)\ar[r,"\deff"] &    \Hrotrot(\Omega) \ar[r,"\rot\rot"]  &  \rmH^{-1}(\Omega) \ar[r] & 0.
\end{tikzcd}
\end{equation}
This is the 2D version of the complex studied in \cite{Chr11NM}, where the Regge elements, defined in detail below, were shown to provide a slightly non-conforming discretization of the middle space. 

The above type of strain complexes have been used to obtain so-called intrinsic formulations in elasticity \cite{CiaCia09,HauHec13,QiuWanZha16}, partly independently of Regge calculus. Indeed, when the displacement is discretized in a $\rmH^1$-conforming space and the range of the deformation operator has been identified as a subspace of some finite element space, one can get an equivalent mixed formulation in terms of the strain (the deformation of the displacement field). A first example in elasticity consists in replacing a formulation with $\rmC^0 \rmP_1$ vector fields as displacements, by a mixed formulation with Regge elements for the strain.

However, in elasticity the case of nearly incompressible materials is important. Just as $\rmC^0 \rmP_1$ displacements can lead to locking in the incompressible limit, Regge elements seem inadequate for mixed formulations in this regime. Compare with \cite[\S 3.3]{HauHec13}.

We introduce the following space, which has regularity similar to the Regge element but with enhanced regularity of the trace: 
\begin{equation}
    \Htr(\Omega) = \{ \sigma \in \Hrotrot(\Omega)  \ : \ \tr \sigma \in \rmH^1(\Omega)\}.
\end{equation}
The proper discretization of this space is one of the main goals of this space. Notice that the trace of the strain is the divergence of the displacement, which is proportional to the pressure. Incompressibility can be expressed as divergence freeness of the displacement (equivalently, the flow of the vectorfield is volume preserving).

We write down some properties of this new space.
\begin{proposition}
The map $\sigma \to (\tr \sigma) \bbI$ sends $\Htr(\Omega)$ to $\rmH^1(\Omega) \bbI$ (by definition) and the latter space is included in $\Htr(\Omega)$.
\end{proposition}

We also use the notation:
\begin{equation}
\Hdivdiv(\Omega) = \{\sigma \in \rmL^2(\Omega) \otimes \bbS_2 \ : \ \div\div \sigma \in \rmH^{-1}(\Omega) \}.
\end{equation}

\begin{proposition} Suppose that $\sigma \in \Hrotrot(\Omega)$.
\begin{equation}\label{eq:stardivdiv}
\sigma \in \Htr(\Omega) \implies \sigma \in \Hdivdiv(\Omega).
\end{equation}
\end{proposition}
\begin{proof}
By the fact that:
\begin{align} 
\div\div \sigma & =   \rot\rot \sigma^{\perp \perp},\\
& = \rot\rot ( \tr (\sigma) \bbI - \sigma) ,\\
& = \Delta \tr (\sigma) - \rot\rot \sigma,
\end{align}
and that $\tr \sigma \in \rmH^1(\Omega)$ implies $\Delta \tr \sigma \in \rmH^{-1}(\Omega)$. 
\end{proof}

\begin{remark}
The computation in the above proof also shows that,  for $\sigma \in \Hrotrot(\Omega) $,  the reciprocal of (\ref{eq:stardivdiv}) holds when $\tr \sigma$ has well defined boundary traces in $\rmH^{1/2}( \partial \Omega)$. Indeed if $v \in \rmL^2(\Omega)$ and $\Delta v \in \rmL^2(\Omega)$ then traces can be defined apriori in $\rmH^{-1/2}(\partial \Omega)$ by Green's identity, and if the traces are actually $\rmH^{1/2}(\partial \Omega)$ then $v \in \rmH^1(\Omega)$.
\end{remark}

Note that the trace enhanced Sobolev space $\Htr(\Omega)$ can be fitted into the following complex:
\begin{equation}
\label{eq:Sobolev-2D}
\begin{tikzcd}
0 \ar[r] &  \Hdiv^1(\Omega)\ar[r,"\deff"] &   \Htr(\Omega) \ar[r,"\rot\rot"]  &  \rmH^{-1}(\Omega) \to 0.
\end{tikzcd}
\end{equation}
We remark that this complex is intermediate between the following two complexes. The less regular \eqref{eq:complex-2d-sobolev}, and the more regular:
\begin{equation}\label{eq:Regge2Dhigh}
\begin{tikzcd}
0 \ar[r] &  \urmH^2(\Omega) \ar[r,"\deff"] &   \rmH^1(\Omega) \otimes \bbS_2 \ar[r,"\rot\rot"]  &  \rmH^{-1}(\Omega) \to 0.
\end{tikzcd}
\end{equation}

The cohomology of \eqref{eq:complex-2d-sobolev} was studied in \cite{HuLinZha25}.  The cohomology of \eqref{eq:Sobolev-2D} can be obtained by similar arguments. Alternatively we can argue as follows:
\begin{proposition}\label{prop:2dinccoh}
The inclusion of \eqref{eq:Sobolev-2D} in \eqref{eq:complex-2d-sobolev} induces isomorphisms on cohomology.
\end{proposition}

\begin{proof}
We apply Lemma \ref{lem:regcoh}. We check the hypothesis:
\begin{itemize}
\item If $u \in \urmH^1(\Omega)$ and $\deff u \in \Htr(\Omega)$, then $\div u = \tr \deff u \in \rmH^1(\Omega)$, so $ u \in \Hdiv^1(\Omega)$.

\item If $u \in \Hrotrot(\Omega) $ then we can choose $v \in \urmH^1(\Omega)$ such that $\div v = \tr u$. Then $u - \deff v \in \Htr(\Omega)$, since it is trace free.
\end{itemize}
This shows that \eqref{eq:Sobolev-2D} has essentially the same cohomology as \eqref{eq:complex-2d-sobolev}. 
\end{proof}

\subsection{Goals and applications\label{sec:goals}}

For applications to the Stokes equation, $\urmH^1(\Omega)$-conforming vector elements that are well behaved with respect to the divergence operator, through projections forming a commuting diagram, have recently been designed. Such spaces constitute the right most part of a de Rham sequence with enhanced regularity, sometimes called a Stokes complex \cite{GuzNei14IMA,GuzNei14MC}. We will use examples defined in \cite{ChrHu18,GuzLisNei22} that are $\Hdiv^1(\Omega)$ conforming, meaning that $\div u \in \rmH^1(\Omega)$ is required.

We will identify a corresponding space of strain tensors that gives a mixed formulation in elasticity that seems well behaved in the incompressible regime. It will be almost $\Htr(\Omega)$ conforming, in the same sense that Regge elements are almost $\Hrotrot(\Omega)$ conforming.

In fact we will identify spaces with even stronger properties. We will construct a finite element space $\Sigma_h$ of symmetric matrix fields, modelling strain tensors, such that for some finite element space of scalars $\Theta_h$ the trace operator $\tr: \Sigma_h \to \Theta_h$ is surjective and we have inclusion of diagonal matrices as a mapping $u \mapsto u \bbI$ sending $\Theta_h \to \Sigma_h$.
Then any operator of the form:
\begin{equation}
\calS(\sigma) = 2\mu \sigma + \lambda (\tr \sigma) \bbI,
\end{equation}
with real parameters $\mu, \lambda$, sends $\Sigma_h$ to $\Sigma_h$. 
In dimension $d\geq 2$, as long as $2\mu + d \lambda \neq 0$ and $\mu \neq 0$, the operator  $\calS$ will be invertible, both continuously and discretely. This should be contrasted with the fact that on Regge elements such an invertibility does not seem to hold, as detailed below.

Furthermore we would like to have a finite element space $\Upsilon_h$ such that we have a commuting diagram:
\begin{equation}
  \begin{tikzcd}
    \vdots \ar[d]\\    
\Upsilon_h \ar[r, "\deff"] \ar[d, "\div"] & \Sigma_h \ar[dl,"\tr"] \ar[r] & \cdots \\
\Theta_h&
    \end{tikzcd}
\end{equation}
The space $\Upsilon_h$ will in general be based on previously constructed Stokes elements \cite{ChrHu18,GuzLisNei22}. Then $\div : \Upsilon_h \to \Theta_h$ is the last part of a smooth de Rham sequence, represented vertically. On the other hand $\deff : \Upsilon_h \to \Sigma_h$ should be the beginning of an elasticity strain complex, represented horizontally.

We have several applications in mind:
\begin{itemize}
 \item In \emph{linear elasticity} the Lam\'e parameters $\mu, \lambda$ describe Hooke's law for homogeneous and isotropic materials. Then $\calS$ maps the strain tensor to the stress tensor. Then $\mu$ is positive and the incompressible limit corresponds to $\lambda \to +\infty$. The basic equation is, for a displacement $u$:
\begin{equation}
    - \div \calS \deff u = f.
\end{equation}
The variational formulation for $u\in \Upsilon_h$ is:
\begin{equation}
\int \calS \deff u : \deff u' = \int f \cdot u'.    
\end{equation}
A primal mixed form of the equation, that isolates the pressure as a separate variable, is, with $u \in \Upsilon_h$ and $p = \lambda \div u \in \Theta_h$:
\begin{align}
    \int  2 \mu \deff u : \deff u' + \int p \div u' & = \int f \cdot u',\\
     \int p' \div u - \frac{1}{\lambda} \int p p' & = 0.
\end{align}
 Projections onto the spaces $\Upsilon_h $ and $\Theta_h$ forming a commuting diagram with respect to the operator $\div : \Upsilon_h \to \Theta_h$ then yields good behavior in the incompressible limit, thanks to uniform inf-sup conditions.

The dual mixed formulation is, in terms of the displacement $u \in \Upsilon_h$ and the stress $\sigma = \calS \deff u \in \Sigma_h$:
\begin{align}
    \int \mathcal{S}^{-1} \sigma : \sigma'  - \int \deff u : \sigma' & = 0,\\
    - \int \deff u' : \sigma & =  \int f \cdot u'. 
\end{align}
We do not require $\Sigma_h$ to be divergence conforming so we cannot integrate the coupling terms by parts while remaining within $\rmL^2(\Omega)$-conforming spaces, contrary to standard stress formulations (see Remark \ref{rem:hellreis} below). In the incompressible limit only the trace free part $\tf \sigma$ of $\sigma$ remains:
 \begin{equation}
     \lim_{\lambda \to \infty} \mathcal{S}^{-1} \sigma = \frac{1}{2\mu} \tf \sigma.  
\end{equation}
In this form, among the Brezzi conditions \cite[\S 5.1.1]{BofBreFor13}, the coercivity on the kernel is not clear. Indeed we lack a priori control of $\tr \sigma$ for divergence free $\sigma$. However the equation is, even on a discrete level and as long as $\lambda < +\infty$, under our hypothesis on the spaces, strictly equivalent to the preceding primal mixed formulation, which is proved to be well behaved.
 
 \item In \emph{complex analysis} the case $\mu = 1/2$ and $\lambda = -1/2$ enables to express the Cauchy-Riemann equations (which say that the tracefree symmetric gradient is $0$):
\begin{equation}
    \calS \deff u = 0.
\end{equation}
 In strong form this equation would exhibit locking in a finite element context. Indeed solutions to the Cauchy-Riemann equations are smooth (in the sense of $\rmC^\infty(\Omega)$), and there are not enough globally smooth functions in $\Upsilon_h$. One would like a weak form of the equation, where $u$ becomes uniquely determined by its boundary values.
 
 \item The case $\mu = 1/2$ and $\lambda = -1$ is relevant to \emph{general relativity} (GR). See \cite[\S 5]{Li18} for the linearized Einstein equations. Notice that in this case $\calS$ does not have a sign. The eigenvalues of $\calS$ are $2 \mu = 1$ and $2\mu + 3\lambda =-2$. For the initial value problem in GR it seems important that the discretization of $\calS$ is stably invertible. One would like the bilinear form:
 \begin{equation}\label{eq:sbil}
 \Sigma_h \times \Sigma_h \ni (\sigma, \sigma') \to \int \calS \sigma : \sigma',
 \end{equation}
 to be stably invertible. Also the constraints would benefit from this property. The stable invertibility seems not to hold when $\Sigma_h$ is the standard Regge element space. The lack of this property was a motivation to study the eigenvalue problem in \cite{Chr11NM}, rather than the time evolution problem. It has been checked numerically that the discrete eigenvalues of $\calS$ on Regge elements (as deduced from (\ref{eq:sbil})) are spread out, including near $0$, rather than being clustered around $1$ and $-2$ \footnote{Personal communication by Douglas N. Arnold.}. On the other hand, under our hypothesis, the discrete eigenvalues are the same as the continuous ones, and $\calS$ is an isomorphism $\Sigma_h \to \Sigma_h$, for any Lamé parameters $\mu, \nu$ such that $2\mu + 3\lambda\neq 0$ and $\mu \neq 0$. In particular the stable invertibility of the above discrete bilinear form holds. 
 
At this point it seems that a good finite element formulation of the initial value problem of GR is not available. This paper is in large part motivated by getting closer to that goal. It illustrates that such finite elements do exist. We aim more for a proof of concept than a systematic theory. 

 \item In \emph{conformal geometry} the situation is, briefly put, that one has a background Riemannian metric $g$, and one works with a conformally transformed metric $g' = \phi g$ for some positive scalar function $\phi = \exp(2u)$. In the Yamabe problem one asks that $g'$ has constant scalar curvature. This gives a non-linear second order elliptic PDE for $\phi$ \cite[\S 14.1]{Tay11III}. In principle this equation can be formulated with standard scalar finite elements for the unknown function $\phi$. But our setting provides discretization spaces that further respect the structure of this problem. Compare with \cite{Luo04}. Indeed, if $g = \bbI \in \Sigma_h $ then for any $\phi \in \Theta_h$ we have $g' = \phi g \in \Sigma_h$. Furthermore the densitized scalar curvature of this $g'$ can be computed even in the non-linear regime \cite{Chr24,GawNeu25}. It will have Dirac deltas in codimension 2 based on angle deficits, and Dirac deltas in codimension 1 based on jumps of mean curvature, in addition to standard terms on elements.  To get an invertible discrete system for the unknown $\phi$, the test space must be of the same dimension as the trial space. Solving the Galerkin problem for $\phi$ on a trial space $\Theta_h$ that fits our hypothesis has the advantage that $\phi \bbI$ is then in a finite element space of metrics that is part of an elasticity complex. Then $\Upsilon_h$ models the reasonable infinitesimal diffeomorphisms the discrete metrics can be subjected to. 
\end{itemize}

\begin{remark}[Hellinger-Reissner]\label{rem:hellreis}
It should also be noted that the mixed formulations that are usually sought in elasticity are based on the Hellinger–Reissner variational principle (see \cite[\S VI.3]{Bra07} and \cite[\S 8.8]{Arn18}), expressed in terms of a stress tensor (not a strain tensor) with $\rmL^2(\Omega)$-conforming divergence. Regge elements are not divergence conforming. They are placed at the beginning of elasticity complexes, not the end. It could be interesting to design stress spaces that are not only divergence conforming but also invariant under $\calS$, but we do not pursue this goal here.
\end{remark}

\section{Regge elements and complexes: old and new}

In this section, we introduce the Regge complex \cite{Chr11NM} and its high order generalization, presented as a discretization of the classical elasticity complex with both Ciarlet type finite elements and distributional ones. 

We first define the high order Regge complex in three dimensions. Let $\lag_k $ be the Lagrange element (piecewise $\rmP_k$ and globally $\rmC^0$), and let $\reg_{k}$ be the Regge element with polynomial degree $k$, as introduced in \cite{Li18}. The degrees of freedom of $\reg_k$ are defined as follows. For a given tetrahedron $T$ any $\sigma \in \rmP_k(T) \otimes \bbS_3$ is determined by: 
\begin{itemize}
\item $\int_{e} \sigma : \rho $	for $\rho  \in \rmP_k(e) \otimes \bbS(e)$, $e \in \calE(T)$
\item $\int_{f} \sigma  : \rho$, for $\rho \in \rmP_{k-1}(f) \otimes \bbS(f)$, $f \in \calF(T)$
\item $\int_T \sigma : \rho$, for $\rho \in \rmP_{k-2}(T) \otimes \bbS_3$
\end{itemize}
Notice that these degrees of freedom act only on the $\tang\tang$-component of $\sigma$ on faces and edges. Simplifications can be obtained for instance from identifications such as $\bbS(e) \approx \bbR$, which shows that the degrees of freedom on the edge $e$ are given by scalar $\rmP_k(e)$ acting on the $\tang\tang$ component of $\sigma$. However such simplifications come at the expense of systematicity of the presentation.

The global continuity imposed on Regge elements is tangential-tangential interelement continuity, which we denote as $\rmC^{\tang\tang}$. We use $\mathring{\lag}_k$ and $\mathring{\reg}_{k}$ to denote the Lagrange element and Regge element with zero boundary condition (i.e. the degrees of freedom on the boundary vanish), respectively. 

For the distributional part of the complexes, we also introduce notations for scalar-, vector- and tensor-valued distributions. We use $\delta_x, \delta_e, \delta_f$ to represent the following (scalar-valued) distributions:
$\langle \delta_x, u \rangle = u(x)$, $ \langle \delta_e, u \rangle = \int_e u$, $\langle \delta_f, u \rangle = \int_f u$, respectively.

Concerning weighted distributions, we define for instance $p\delta_e$ by
$\langle p\delta_e, u \rangle = \int_e pu,$ where $p$ is usually a polynomial defined on the edge $e$. Similarly, we define the weighted face moment $p\delta_f$ and the weighted cell moment $p\delta_T$.

Distributions may also be tensor valued. Actually all the distributions we consider here are measures. Measures on $\overline{\Omega}$ with support on $\partial \Omega$ are excluded, to obtain distributions on $\Omega$ (acting on compactly supported test-functions).

We define 
\begin{equation} \label{def:Qstar}
  \mathring{\lag}^\ast_{k}(\calT) = \bigoplus_{x \in \calV^{\circ}} \bbR \delta_x \oplus \bigoplus_{e \in \calE^{\circ}} \rmP_{k-2}(e) \delta_e \oplus \bigoplus_{f \in \calF^{\circ}} \rmP_{k-3}(f)\delta_f \oplus \bigoplus_{T \in \calT} \rmP_{k-4}(T)\delta_T,	
\end{equation}
in  three dimensions. This distributional space corresponds to the standard degrees of freedom of the Lagrange space $\mathring{\lag}_k$ with zero boundary condition. There is a canonical invertible bilinear form on $\mathring{\lag}_k \times \lag^\ast_k$ extending the $\rmL^2$ product on smooth functions. Therefore, the space $\lag^\ast_{k}$ can be regarded as the dual space of $\mathring{\lag}_k$. 

Next, we introduce some tensor-valued distributional spaces. The aim is to precisely characterize the dual space of the Regge element $\mathring{\reg}_k$. We define 
\begin{equation}
\label{def:Xistar}
	\mathring{\reg}_{k}^{\ast}(\calT) = \bigoplus_{e \in \calE^{\circ}}  \rmP_{k}(e) \bbS(e) \delta_e \oplus \bigoplus_{f \in \calF^{\circ}} \rmP_{k-1}(f)  \bbS(f) \delta_f \oplus \bigoplus_{T\in\calT} \rmP_{k-2}(T) \bbS_3 \delta_T.
\end{equation}

Similar to the Lagrange case, it can be seen that the tensor valued distributional space $\mathring{\reg}^\ast_k$ is a dual space of $\mathring{\reg}_k$ with respect to an extension of $\rmL^2$ duality. When $k = 0$, only the edge tangential-tangential deltas remain in the expression of \eqref{def:Xistar}.

The original 3D Regge complex makes use of the lowest order versions of both the Lagrange space and the Regge space, along with their dual distributional spaces. It is formulated as 
\begin{equation}
  \label{eq:regge-3d}
  \begin{tikzcd}
    0 \ar[r] & \lag_1(\calT) \otimes \bbR^3 \ar[r,"\deff"] &  \reg_0(\calT) \ar[r,"\inc"]  & \mathring{\reg}^{\ast}_0(\calT) \ar[r,"\div"]  & \mathring{\lag}^{\ast}_{1}(\calT) \otimes \bbR^3 \ar[r] &  0.
  \end{tikzcd}
\end{equation}
There are no boundary conditions in the first two spaces. In the last two spaces there are no distributions on the boundary.

In \cite{ChrHuLin26}, it is shown that the cohomology of \eqref{eq:regge-3d} is essentially the same as that of the smooth elasticity complex. By BGG techniques \cite{ArnHu21} it is isomorphic to $ \CdR^{\bs}(\Omega) \otimes \RM_3$. 

The above complex can be generalized to arbitrary order, retaining the duality structure of the low order case. 
\begin{theorem}
For $k \geq 0$, the following spaces and operators form a complex. 
\begin{equation}
\label{eq:regge-3d-k}
\begin{tikzcd}
0 \ar[r] & \lag_{k+1}(\calT) \otimes \bbR^3 \ar[r,"\deff"] &  \reg_{k}(\calT) \ar[r,"\inc"]  & \mathring{\reg}^{\ast}_k(\calT) \ar[r,"\div"]  & \mathring{\lag}^{\ast}_{k+1}(\calT) \otimes \bbR^3 \ar[r] &  0.
\end{tikzcd}
\end{equation}
\end{theorem}

The only non-trivial mapping property is that related to the $\inc$ operator, and is guaranteed by the following proposition, for which we refer to \cite[Theorem 4.2]{CheHua22}. The edge term, which is the only one in the lowest order case, was identified in \cite[Lemma 1]{Chr11NM}.
\begin{proposition}
\label{prop:inc3D-formula}
Suppose $\sigma \in \rmL^2(\calT) \otimes \bbS_3$ is piecewise smooth, and has tangential-tang\-en\-tial interelement continuity. Then, it holds that 
\begin{equation}
  \inc \sigma = \inc_h \sigma + \text{ face term } + \text{ edge term},
\end{equation}
with: 
\begin{equation}
  \text{face term } = \sum_{f \in \calF^{\circ}} \Pi_f \jump{\curl \sigma}_f \times \norm_f + \grad_f (\Pi_f \jump{\sigma}_f \norm_f),
\end{equation}
and:
\begin{equation}
  \text{edge term } = \sum_{e \in \calE^{\circ}} \left( \sum_{f \in \calF^\circ : e \in f} \tang_f^e \cdot \jump{\sigma}_f  \norm_f \right) \tang_e^\transp \tang_e \delta_e.
\end{equation}  
Here $\tang_f^e = \calO (e,f) \bm (\tang_e \times \norm_f)$ is the normal to the edge $e$ pointing into the face $f$. The jump $\jump{ \sigma}_f $ across $f$ takes into account its orientation.  
\end{proposition}

\begin{remark}
The two first spaces have natural degrees of freedom, and for the two distributional spaces one can consider the preduals as defining degrees of freedom. The corresponding projections commute with the differential operators.
\end{remark}

We now elaborate on the two dimensional case. Then the lowest order Regge complex is expressed as
\begin{equation}
  \label{eq:regge-2d}
  \begin{tikzcd}
    0 \ar[r] & \lag_1(\calT) \otimes \bbR^2 \ar[r,"\deff"] & \reg_0(\calT) \ar[r,"\rot\rot"]  &  \lag^{\ast}_{1}(\calT)  \ar[r] & 0.
  \end{tikzcd}
\end{equation}
It was shown in \cite{HuLinZha25} that the cohomology of \eqref{eq:regge-2d} is isomorphic to $\CdR^\bs(\Omega) \otimes \RM_2$. 
The high order version is:
\begin{theorem}
\label{thm:high-regge}
We have the well defined complex:
\begin{equation}
   \label{eq:regge-2d-k}
\begin{tikzcd}
0 \ar[r] & \lag_{k+1}(\calT) \otimes \bbR^2 \ar[r,"\deff"] &  \reg_k(\calT) \ar[r,"\rot\rot"]  &  \lag^{\ast}_{k+1}(\calT)  \ar[r] &  0.
\end{tikzcd}
\end{equation}
Its cohomology is isomorphic to $\CdR^\bs(\Omega) \otimes \RM_2$.
\end{theorem}

The complex \eqref{eq:regge-2d-k} with $k=1$ is represented in Figure \ref{fig:reg2complex}. 

\begin{figure}[htb]
\begin{tikzpicture}[x=1cm,y=1cm,yscale=1,isosceles triangle stretches= true]

\node[draw,isosceles triangle,
      minimum width=3.3692393cm,
      minimum height=2.9132936cm, anchor = lower side, shape border rotate = 90] at (1.7431774,-6.0145435) {};

\node[draw,isosceles triangle,
      minimum width=3.3692393cm,
      minimum height=2.9132936cm, anchor = lower side, shape border rotate = 90] at (7.63588,-6.07125) {};

\node[draw,isosceles triangle,
      minimum width=3.3692393cm,
      minimum height=2.9132936cm, anchor = lower side, shape border rotate = 90] at (13.55588,-6.21125) {};

\draw[->,line width=0.04cm] (3.8368013,-4.569627) -- (5.3758736,-4.595785);
\draw[->,line width=0.04cm] (9.729523,-4.6252356) -- (11.268558,-4.653589);

\fill (1.7385577,-3.15125) circle[radius=0.08cm];
\fill (1.7385577,-6.03125) circle[radius=0.08cm];
\fill (0.09855774,-6.03125) circle[radius=0.08cm];
\fill (3.3985577,-6.03125) circle[radius=0.08cm];
\fill (0.97855777,-4.47125) circle[radius=0.08cm];
\fill (2.5585577,-4.47125) circle[radius=0.08cm];

\draw[line width=0.04cm] (8.898558,-4.97125) -- (9.198558,-5.49125);
\draw[line width=0.04cm] (7.158558,-3.59125) -- (6.8185577,-4.13125);
\draw[line width=0.04cm] (6.34598,-4.903644) -- (5.9911356,-5.5188556);
\draw[line width=0.04cm] (6.5585575,-6.21125) -- (7.2985578,-6.21125);
\draw[line width=0.04cm] (7.9585576,-6.23125) -- (8.618558,-6.23125);
\draw[line width=0.04cm] (8.094219,-3.59125) -- (8.422896,-4.085589);

\fill (7.598558,-4.51125) circle[radius=0.08cm];
\fill (7.218558,-5.17125) circle[radius=0.08cm];
\fill (7.9785576,-5.17125) circle[radius=0.08cm];

\fill[lightgray] (13.538558,-3.33125) rectangle ++(0.16cm,0.16cm);
\fill[lightgray] (13.538558,-6.21125) rectangle ++(0.16cm,0.16cm);
\fill[lightgray] (11.898558,-6.21125) rectangle ++(0.16cm,0.16cm);
\fill[lightgray] (15.198558,-6.21125) rectangle ++(0.16cm,0.16cm);
\fill[lightgray] (12.778558,-4.65125) rectangle ++(0.16cm,0.16cm);
\fill[lightgray] (14.358558,-4.65125) rectangle ++(0.16cm,0.16cm);

\node[anchor=base west] at (2.5785577,-2.95125) {$\otimes \mathbb{R}^2$};
\node[anchor=base west] at (4.1685576,-4.23125) {$\deff$};
\node[anchor=base west] at (10.128558,-4.23125) {$\rot \rot$};

\end{tikzpicture}
\caption{Regge complex \eqref{eq:regge-2d-k} with $k=1$ (starting with $\lag_2 \otimes \bbR^2$)}
\label{fig:reg2complex}
\end{figure}
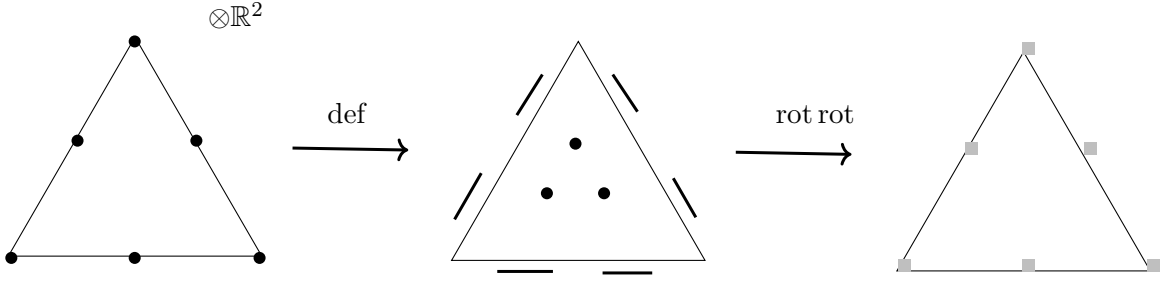

We proceed with the proof in several steps. To check that it's a complex, it suffices to show that it is closed under $\rot\rot$ (meaning that $\rot\rot \reg_k \subseteq \lag^{\ast}_{k+1}$) using the next proposition. 

\begin{proposition}
\label{prop:inc2D-formula}
Suppose that $\sigma \in \rmL^2(\calT) \otimes \bbS_2$ is piecewise smooth, and has tang\-en\-tial-tang\-en\-tial interelement continuity. Then, it holds that 
\begin{equation}
  \rot\rot \sigma = \rot_h\rot_h \sigma + \sum_{e \in \calE^{\circ}} \left( \jump{ \rot_h \sigma \cdot \tang_e} - \frac{\partial}{\partial \tang}\jump{\norm_e^\transp \sigma \tang_e} \right) \delta_e + \text{ vertex term}.
\end{equation}

Here, the vertex term is 
\begin{equation}
  \sum_{x\in \calV^{\circ}}\left(  \sum_{x\in e} \calO(x,e) \norm_e^\transp \jump{\sigma}_e \tang_e \right) \delta_x.
\end{equation}
 As a consequence, when $\sigma$ is continuous at vertices, the above vertex term vanishes.
\end{proposition}

\begin{proof}
By direct calculation:
\begin{align}
  & \langle \rot\rot \sigma, \varphi \rangle\\
  & =  \int_{\calT} \sigma : \rot\rot \varphi,  \\
	& = \int_{\calT} \rot_h \sigma \cdot \transp \rot\varphi + \sum_{e \in \calE^{\circ}}\jump{ \sigma \cdot \tang_e } \cdot (\transp \rot \varphi)\\ 
	& = \int_{\calT} \rot_h \sigma \cdot \transp \rot\varphi + \sum_{e \in \calE^{\circ}}\jump{ \norm_e^\transp \sigma \tang_e }  (\transp \rot \varphi \cdot \norm_e) \\
	& = \int_{\calT} \rot_h \rot_h \sigma \cdot \varphi + \sum_{e \in \calE^{\circ}}\jump{ \rot_h \sigma \cdot \tang_e}\varphi + \sum_{e \in \calE^{\circ}}\jump{ \norm^\transp \sigma \tang_e }  \frac{\partial}{\partial \tang_e}\varphi \\ 
	& = \int_{\calT} \rot_h \rot_h \sigma \cdot \varphi  + \sum_{e \in \calE^{\circ}} \left(\jump{ \rot_h \sigma \cdot \tang_e} - \frac{\partial}{\partial \tang_e}\jump{ \norm_e^\transp \sigma \tang_e } \right) \varphi + \text{ vertex term}.
\end{align}
Here the vertex term is as announced.
\end{proof}

Concerning the cohomological result stated in \Cref{thm:high-regge} we postpone the proof to \Cref{sec:cohhess2d}.

\begin{remark}
  If $f \in \mathring{\lag}_{k+1}(\calT)$ then $\rot \rot f$ is a distribution in the space of degrees of freedom of $\reg_k(\calT)$. Commutativity of natural projections follows.
\end{remark}

\section{Trace enhanced Regge elements in dimension two\label{sec:dimtwo}}

\subsection{Trace enhanced elements}

In this subsection we discuss the construction of $\Htr(\Omega)$ conforming finite elements,  as trace enhanced Regge elements. 

We first recall the conformity condition. Any piecewise smooth metric is $\Htr(\Omega)$ conforming if and only both its tangential-tangential component and normal-normal components are continuous across the edges. Then the trace is also continuous at vertices. This continuity condition motivates us to define the following bubble space. 

We fix the polynomial degree $k \geq 1$. Define the bubble space attached to a face $f$ as
\begin{equation}
  \uuB(f) = \{ \sigma \in \rmP_k(f) \otimes \bbS_2 \ : \  \sigma_{\tang\tang} = 0 \text{ on } \partial f \text{ and } \sigma_{\norm \norm} = 0 \text{ on } \partial f\},
\end{equation}
and the bubble space with vanishing values at vertices as
\begin{equation}
\huuB(f) = \{\sigma \in \uuB(f) \ : \ \sigma(x) = 0 \text{ for all } x \in \calV(f)\}.
\end{equation}
\begin{remark}
\label{rmk:right-angle}
If none of the angles of $f$ is right, then it holds that $\uuB(f) = \huuB(f)$.	
\end{remark}

\begin{proposition}
\label{prop:bubble2D}
We have the following direct sum decomposition of $\huuB(f)$. For $\sigma \in \huuB(f)$, it holds that
\begin{equation}
\label{eq:bubble2D}
\sigma = \lambda_0 \lambda_1 \lambda_2  \rho \ + \sum_{i=0}^2  \lambda_{i}\lambda_{i+1} p_i \sym(\norm_i \tang_i^\transp),
\end{equation}
where $\rho \in \rmP_{k-3}(f) \otimes \bbS_2$ and $p_i \in \rmP_{k-2}(e_i)$.  Here $\norm_i$ are the unit normal vectors and $\tang_i$ are the unit tangential vectors of the edge $e_i$, joining vertex $i$ to $i+1$, for $i = 0,1,2$ with 3-periodic indices. The polynomial $p_i$ is extended from $e_i$ to $f$ by the Bernstein basis of $\rmP_{k-2}(e_i)$, expressed with $\lambda_i, \lambda_{i+1}$.

\end{proposition}

\begin{proof}
Clearly, the right-hand side is a direct sum and is included in $\huuB(f)$.

Now, suppose that $\sigma \in \huuB(f)$. We consider the restriction $\sigma|_{e_i}$. Since $\norm_i \norm_i^\transp$, $\tang_i \tang_i^\transp$ and $\sym(\norm_i \tang_i^\transp)$ form a basis of $\bbS_2$, we obtain 
$\sigma|_{ e_i} = q_i \sym(\norm_i \tang_i^\transp)$ for some polynomial $q_i$ of degree $k$. 
The vanishing condition at the vertex then yields that  $\sigma|_{e_i} = \lambda_{i}\lambda_{i+1} p_i \sym(\norm_i \tang_i^\transp)$ for some polynomials $p_i \in \rmP_{k-2}(e_i)$. These polynomials are extended to $f$ through the Bernstein basis.

Finally, consider $\sigma - \sum_{i} \lambda_{i+1}\lambda_{i+2} p_i \sym(\norm_i \tang_i^\transp)$. Since all components vanish on the boundary, we conclude that it is of the form $\lambda_1\lambda_2\lambda_3 \rho$ for some $ \rho \in \rmP_{k-3}(f) \otimes \bbS_2$.
\end{proof}

Now we consider the construction a finite element space $\Sigma_{k}$ of symmetric tensors. We require $k \geq 1$. The local shape function space is $\rmP_k(f) \otimes \bbS_2$, and for $\sigma \in \rmP_k (f) \otimes \bbS_2$, the degrees of freedom are defined as:
\begin{itemize}
\item $\sigma(x)$ at each vertex $x \in \calV(f)$
\item $\int_{e} \sigma_{\tang\tang} p$ and $\int_{e} \sigma_{\norm \norm} p$ for $p \in \rmP_{k-2}(e)$, for each edge $e \in \calE(f)$
\item $\int_{f} \sigma : b $ for $ b \in \huuB(f)$
\end{itemize}

\begin{proposition}
  The above degrees of freedom are unisolvent on $\rmP_{k}(f) \otimes \bbS_2$.\\
  The dimension of the space $\Sigma_{k}$ is $3|\calV| + 2(k-1) |\calE| + \frac{3}{2}(k-1)(k) |\calF|.$
\end{proposition}

\begin{proof}
The only non-trivial part is the dimension count. By the characterization of $\huuB(f)$, it holds that 
\begin{equation}
  \dim \huuB(f) = 3 \binom{k-1}{2} + 3 (k-1),
\end{equation}
and the total number of degrees of freedom is 	
\begin{equation}
  9 + 6(k-1) + \frac{3}{2}(k-1)(k-2) + 3(k-1) = \frac{3}{2}(k+2)(k+1),
  \end{equation}
which is equal to the dimension of $\rmP_k(f) \otimes \bbS_2$.
\end{proof}

The finite element $\Sigma_k$ imposes full continuity at vertices. In view of possible applications to General Relativity it seems best to avoid requiring full continuity of tensors at vertices. Indeed, if a simplicial complex is equipped with a background Regge metric, it seems possible to impose that another Regge metric has a given relative trace at vertices, but a full continuity of tensors seems meaningless in the presence of distributional curvature of the background, as there is no well defined tangent space.

We therefore proceed to construct a less regular trace enhanced Regge element space. The finite element will have continuity only of the trace at vertices, not full continuity. Referring to \Cref{rmk:right-angle}, we require that there are no right angles at vertices so that $\uuB(f) = \huuB(f)$.

We now define the low regularity element $\overline{\Sigma}_{k}$. For $\sigma \in \rmP_k(f) \otimes \bbS_2$, where $k \geq 1$, the degrees of freedom are
\begin{itemize}
\item $\tr \sigma(x)$ for $x \in \calV(f)$ 
\item $\int_e \sigma_{\tang\tang} p$ for $p \in \rmP_k(e)$ for $e \in \calE(f)$
\item $\int_e \sigma_{\norm \norm} p$ for $p \in \rmP_{k-2}(e)$ for $e \in \calE(f)$
\item $\int_f \sigma : b$ for $b \in \uuB(f) = \huuB(f)$
\end{itemize}
Based on the previous result it is easy to see that the set of degrees of freedom is unisolvent. 

For convenience we state the lowest order case:
\begin{proposition}
\label{prop:2d-low}
Suppose that all the angles at vertices in $\calT$ are different from $\pi/2$. The local space $\rmP_1 (f) \otimes \bbS_2$ is unisolvent with respect to the above degrees of freedom, namely traces at vertices and two tangential-tangential moments per edge. 
\end{proposition}

\begin{proof}
Suppose $\sigma \in \rmP_1(f) \otimes \bbS$ is such that the above degrees of freedom vanish. Then $\tr \sigma$ which is $\rmP_1(f)$ and $0$ at vertices, is $0$ on $f$. Furthermore $\sigma_{\tang\tang} = 0$ on the boundary $\partial f$. Therefore, it follows from \Cref{prop:bubble2D} that $\sigma = 0$. 
\end{proof}

\begin{remark}
  The angle condition is one aspect of the fact that this finite element is not affine invariant.
\end{remark}

Next, we give the detailed expression of $\rot\rot \sigma$ when $\sigma$ has tangential-tangential continuity. We compare the cases where $\sigma $ has full or only trace continuity at vertices.

\begin{proposition}
The following results hold.
\begin{itemize}
\item $\Sigma_k$ has vertex continuity, therefore $\rot\rot \Sigma_k$ lies in the space:
  \begin{equation}
    \bigoplus_{f \in \calF} \rmP_{k-2}(f)\delta_f \oplus \bigoplus_{e \in \calE^\circ} \rmP_{k-1} (e)\delta_e.
  \end{equation}
  
\item $\overline{\Sigma}_k$ does not have vertex continuity, therefore $\rot\rot \overline{\Sigma}_k$ lies in the space
\begin{equation}
  \mathring{\lag}_{k+1}^{\ast} =  \bigoplus_{f \in \calF} \rmP_{k-2}(f)\delta_f \oplus \bigoplus_{e \in \calE^{\circ} } \rmP_{k-1} (e)\delta_e \oplus \bigoplus_{x \in \calV^\circ} \bbR \delta_x.
\end{equation}
  
\end{itemize}	
\end{proposition}

The dimensions of Lagrange spaces (and their degrees of freedom) are standard, e.g.:
\begin{equation}
  \dim \mathring{\lag}_{k+1}^{\ast} = \frac{3}{2}(k-1)k|\calF| + k|\calE^{\circ}| + |\calV^{\circ}|.
\end{equation}

\begin{remark}\label{rmk:edge-delta}	
According to the trace theorem, an edge delta is $\rmH^{-1}(\Omega)$ conforming in the sense that $\rmH^1(\Omega)$ functions have well-defined traces on edges, for instance in $\rmL^2(e)$. Consequently, $\rot \rot : \Sigma_k \to \rmH^{-1}(\Omega)$, and the element is thus truly conforming. This might lead to some new results, such as simplifications in the analysis of TDNNS formulations \cite{PecSch11}.
\end{remark}

At this point we do not claim surjectivity of $\rot \rot: \overline{\Sigma}_k \to \mathring{\lag}_{k+1}^{\ast}$. Indeed, while we proved unisolvence for $k \geq 1$, we have further requirements to obtain a good discrete complex, as detailed in the next subsection.

\begin{remark} For both $\Sigma_k$ and $\overline{\Sigma}_k$ the trace is in $\rmC^0 \rmP_k$ and in turn $\rmC^0\rmP_k \bbI$ is included in these spaces. Thus they fullfill the requirement indicated in the introduction.
\end{remark}

\subsection{Complexes built on trace enhanced elements}
We now consider the discretization of the complex \eqref{eq:Sobolev-2D} based on the above finite element space $\Sigma_k$. It will have the following form:
\begin{equation}
\label{eq:discrete-2D}
\begin{tikzcd}
0 \ar[r] & \Upsilon_{k+1} \ar[r,"\deff"] &  \Sigma_{k} \ar[r,"\rot\rot"] & \mathring{\lag}_{k+1}^{\ast} \ar[r] &  0.
\end{tikzcd}
\end{equation}

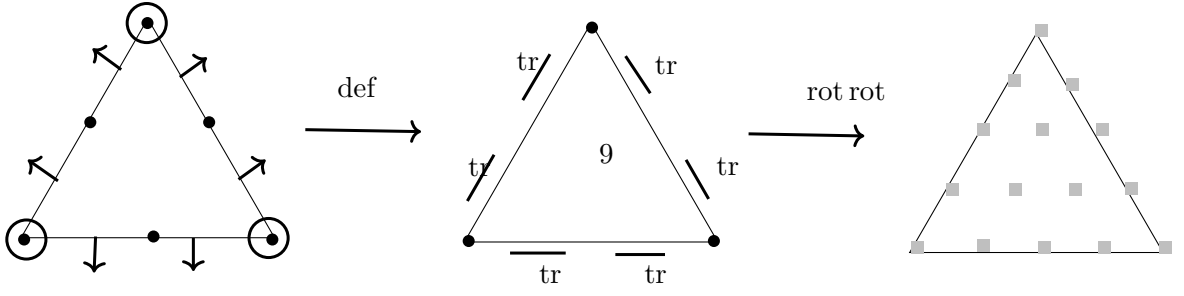
\begin{figure}[htb]
  \begin{tikzpicture}[x=1cm,y=1cm,yscale=1,isosceles triangle stretches= true]

\node[draw,isosceles triangle,
      minimum width=3.3692393cm,
      minimum height=2.9132936cm, anchor = lower side, shape border rotate = 90] at (1.8746197,-5.9232936) {};
\node[draw,isosceles triangle,
      minimum width=3.3692393cm,
      minimum height=2.9132936cm, anchor = lower side, shape border rotate = 90] at (7.767322,-5.98) {};
\node[draw,isosceles triangle,
      minimum width=3.3692393cm,
      minimum height=2.9132936cm, anchor = lower side, shape border rotate = 90] at (13.687323,-6.12) {};

\draw[->,line width=0.04cm] (3.9682436,-4.4783764) -- (5.507316,-4.504535);
\draw[->,line width=0.04cm] (9.860965,-4.5339856) -- (11.4,-4.562339);

\foreach \x/\y in {
  1.87/-3.06,
  0.23/-5.94,
  3.53/-5.94,
  1.11/-4.38,
  2.69/-4.38,
  1.95/-5.9,
  7.78/-3.12,
  6.14/-5.96,
  9.4/-5.96
}{
  \fill (\x,\y) circle[radius=0.08cm];
}

\draw[line width=0.04cm] (9.03,-4.88) -- (9.33,-5.4);
\draw[line width=0.04cm] (6.477422,-4.812394) -- (6.1225777,-5.4276056);
\draw[line width=0.04cm] (6.69,-6.12) -- (7.43,-6.12);
\draw[line width=0.04cm] (8.09,-6.14) -- (8.75,-6.14);
\draw[line width=0.04cm] (8.225661,-3.5) -- (8.554338,-3.994339);
\draw[line width=0.04cm] (7.2207847,-3.474358) -- (6.8592153,-4.085642);

\foreach \x/\y in {
  13.67/-3.24,
  12.03/-6.12,
  15.33/-6.12,
  12.91/-4.56,
  14.49/-4.56,
  13.71/-6.12,
  13.31/-3.9,
  12.49/-5.36,
  14.09/-3.96,
  14.87/-5.34,
  12.91/-6.1,
  14.51/-6.12,
  13.69/-4.56,
  13.33/-5.36,
  14.13/-5.36
}{
  \fill[lightgray] (\x,\y) rectangle ++(0.16cm,0.16cm);
}

\node[anchor=base west] at (4.26,-4.04) {$\deff$};
\node[anchor=base west] at (10.5,-4.08) {$\rot \rot$};

\draw[line width = 0.04cm, ->] (1.5202447,-3.6796672) -- (1.1397554,-3.4003327);
\draw[line width = 0.04cm, ->] (0.68,-5.16) -- (0.28,-4.88);
\draw[line width = 0.04cm, ->] (2.3,-3.78) -- (2.66,-3.52);
\draw[line width = 0.04cm, ->] (2.48,-5.92) -- (2.48,-6.36);
\draw[line width = 0.04cm, ->] (3.08,-5.14) -- (3.44,-4.88);
\draw[line width = 0.04cm, ->] (1.18,-5.9) -- (1.16,-6.38);

\draw[line width=0.04cm] (1.86,-3.06) circle[radius=0.26];
\draw[line width=0.04cm] (0.26,-5.94) circle[radius=0.26];
\draw[line width=0.04cm] (3.48,-5.92) circle[radius=0.26];

\node[anchor=base west] at (6.64,-3.68) {tr};
\node[anchor=base west] at (6.0,-5.1) {tr};
\node[anchor=base west] at (6.94,-6.52) {tr};
\node[anchor=base west] at (8.34,-6.52) {tr};
\node[anchor=base west] at (8.48,-3.74) {tr};
\node[anchor=base west] at (9.3,-5.08) {tr};
\node[anchor=base west] at (7.74,-4.92) {9};

\end{tikzpicture}
  \caption{The trace enhanced Regge complex \eqref{eq:discrete-2D} with $k=3$.}
  \label{fig:disc2d3k}
\end{figure}

In this subsection, we require $k \geq 3$. The degrees of freedom of the $\Hdiv^1(\Omega)$ conforming space $\Upsilon_{k+1}$ are defined as follows, see \cite[\S 2.3]{HuMaZha21} and \cite[\S 6.3]{HuLinWu24}. We require the following bubble space:
\begin{equation}
  \uB_k(f) := \{ v \in \rmP_k(f) \otimes \bbR^2 \ : \  v|_{\partial f} = 0 ,\ \grad v(x) = 0 \quad \forall x \in \calV(f)\}.
\end{equation}
For $ u \in \rmP_{k+1}(f) \otimes \bbR^2$ we evaluate:
\begin{itemize}
\item $u(x), \grad u(x)$ at each vertex $x \in \calV(f)$
\item $\int_{e} u \cdot p$ for $p \in \rmP_{k-3}(e) \otimes \bbR^2, e \in \calE(f)$
\item $\int_{e} (\div u) p $ for $p \in \rmP_{k-2}(e), e \in \calE(f)$
\item $\int_{f}  u \cdot b$ for $b \in \uB_k(f)$.
\end{itemize}
In \cite{HuMaZha21,HuLinWu24} it is computed that $\dim \uB_k(f) = (k-2)^2-1$, and unisolvence is proved. The space $ \Upsilon_{k+1}$ is a subspace of a vector Hermite space, defined as:
\begin{align}
\Upsilon_{k+1} = \{ & v \in \rmL^2(\Omega) \otimes \bbR^2 \  : \   v|_f \in \rmP_{k+1}(f) \otimes \bbR^2 \text{ on elements } f \in \calF, \\
&  v,\ \div  v \text{ are continuous on } \Omega, \\
& \grad v \text{ is single-valued at vertices}\}.
\end{align}
In particular $ \Upsilon_{k+1}$ is $\Hdiv^1(\Omega)$ conforming. It follows that the we have the dimension count:
\begin{equation}
  \dim  \Upsilon_{k+1} = 6|\calV| + (3k-5) |\calE| + (k-1)(k-3)|\calF|.
\end{equation}
We notice that the complex \eqref{eq:discrete-2D} is indeed closed under the operators.

The case $k=3$ of the complex is represented in Figure \ref{fig:disc2d3k}.

We now check the characteristic of the complex \eqref{eq:discrete-2D}. The Euler characteristic of the domain, is $\chi = |\calV| - |\calE|+ |\calF|$. If the domain $\overline \Omega$ is contractible then $\chi = 1$.

The dimension of the spaces are given as follows.
\begin{align}
\dim \Upsilon_{k+1} & = 6|\calV| + (3k-5) |\calE| + (k-1)(k-3)|\calF| , \\
\dim \Sigma_{k} &= 3|\calV| + 2(k-1) |\calE| + \frac{3}{2}(k-1)(k) |\calF|, \\
\dim \mathring{\lag}_{k+1}^{\ast} &= \frac{1}{2}(k-1)k|\calF| + k|\calE^{\circ}|.
\end{align}
We compute:
\begin{align} & \dim  \Upsilon_{k+1}- \dim \Sigma_{k} + \dim \mathring{\lag}_{k+1}^{\ast}\\
  &=  3|\calV| +(k-3)|\calE| + k|\calE^{\circ}| - (3k-3)|\calF|, \\
& = 3 \chi + k|\calE| + k|\calE^{\circ}| - 3k|\calF|, \\ 
& = 3 \chi.
\end{align}
We used that $|\calE | + |\calE^\circ| = 3|\calF|$. We notice that in this last formula $3 = \dim \RM_2$. On a topologically trivial domain, the finite element complex has cohomology $\RM_2$ at index 0 as it should, and is exact at index $1$, as it also should. The computation of the characteristic is then enough to check that it is also exact at index $2$, as it should. In other words we have checked that the operator $\rot \rot : \Sigma_k \to \mathring{\lag}_{k+1}^{\ast}$ is surjective.

By the results in \Cref{sec:cohomology} we can obtain the following more precise cohomological result.
\begin{theorem}
\label{thm:2d-main}
The cohomology of \eqref{eq:discrete-2D} is $\CdR^\bs(\Omega) \otimes \RM_2 $. 
\end{theorem}

\shc{Can we define commuting projections?}

\shc{Complexes built on $\overline{\Sigma}_k$ instead of $\Sigma_k$?}

\subsection{Trace enhanced Regge complexes on the Hsieh--Clough--Tocher split}

In this section, we define a discrete complex on the Hsieh--Clough--Tocher (HCT) split.  While the previous complexes were defined only for $k \geq 3$, the HCT split enables us to lower the polynomial degree. We obtain a discrete complex centered on the previously defined space $\overline{\Sigma}_k$, which exhibits only partial continuity at the vertices. We treat the case $k=1$.

To achieve this, we begin with a general triangulation $\mathcal{T}$. Then, we add an inpoint to each face $f \in \calF$ to get an HCT split. The sets of vertices, edges, and faces of the HCT refinement are respectively denoted as $\mathcal{V}^{CT}$, $\mathcal{E}^{CT}$, and $\mathcal{F}^{CT}$. The refined mesh will be denoted by $\mathcal{T}^{CT}$. For every triangular face $f \in \mathcal{F}$, we denote the newly added vertex (inpoint) as $x_f$ and by $f^{CT}$ its refined mesh. In the CT split, the condition that vertex angles should be different from $\pi/2$ seems easy to satisfy. 

We define a complex:
\begin{equation}
\label{eq:dicsrete-2D-low}
\begin{tikzcd}
0 \ar[r] &  \Upsilon^{CT}(\calT) \ar[r,"\deff"] &  \overline{\Sigma}(\calT^{CT})  \ar[r,"\rot\rot"]  & \lag_{2}^{\ast}(\calT^{CT}) \ar[r] &  0.
\end{tikzcd}
\end{equation}
Here, $\Upsilon^{CT}$ is  the $\Hdiv(\Omega)$ conforming space on $\calT$,  proposed in  \cite{ChrHu18}. The shape function space is 
\begin{equation}
  \Upsilon^{CT}(f) = \{ u \in \lag_2(f^{CT}) \otimes \bbR^2 \ : \ \div u \in \lag_1(f^{CT})\}.
  \end{equation}
The degrees of freedom are:
\begin{itemize}
\item $u(x)$, $\div u(x)$ for each vertex $x \in \calV(f)$
\item $\int_{e} u$  on each edge $e \in \calE(f)$ (two components per edge). 
\end{itemize}
The unisolvence was shown in \cite[\S 4]{ChrHu18}. The local dimension is 15, and the global dimension is $3|\calV|+2|\calE|$. Notice that at vertices we do not impose the full $\rmC^1$ regularity of the vector field. Only continuity of the vectorfield and its divergence are required.
 
The dimension of each space is given as follows:
\begin{align}
\dim \Upsilon^{CT} & = 3|\calV| + 2 |\calE|  , \\
\dim \overline{\Sigma}(\calT^{CT})  &= |\calV^{CT}| + 2 |\calE^{CT}| = |\calV| + 2|\calE| + 7|\calF|, \\
\dim \lag_{2}^{\ast} &= |\calV^{CT,\circ}| +  |\calE^{CT,\circ}| =  |\calV^{\circ}| + |\calE^{\circ}| + 4|\calF|.
\end{align}
The complex is represented in Figure \ref{fig:complex2dhct}.
\begin{figure}[htb]
  \begin{tikzpicture}[x=1cm,y=1cm,yscale=1,isosceles triangle stretches= true]

\definecolor{colour0}{rgb}{0.7019608,0.7019608,0.7019608}

\node[draw,isosceles triangle,
      minimum width=3.3692393cm,
      minimum height=2.9132936cm,anchor = lower side, shape border rotate = 90] at (1.7431774,-5.8132935) {};

\node[draw,isosceles triangle,
      minimum width=3.3692393cm,
      minimum height=2.9132936cm,anchor = lower side, shape border rotate = 90] at (7.443177,-6.0332937) {};

\node[draw,isosceles triangle,
      minimum width=3.3692393cm,
      minimum height=2.9132936cm,anchor = lower side, shape border rotate = 90] at (13.623178,-6.0732937) {};

\draw[->,line width=0.04cm] (3.8368013,-4.3683767) -- (5.3758736,-4.394535);
\draw[->,line width=0.04cm] (9.729523,-4.4239855) -- (11.268558,-4.4523387);

\fill (1.7385577,-2.95) circle[radius=0.08cm];
\fill (0.09855774,-5.83) circle[radius=0.08cm];
\fill (3.3985577,-5.83) circle[radius=0.08cm];
\fill (0.9585577,-4.27) circle[radius=0.08cm];
\fill (2.5385578,-4.29) circle[radius=0.08cm];
\fill (1.7385577,-5.83) circle[radius=0.08cm];

\node[anchor=base west] at (1.8485577,-2.75) {$\div$};
\node[anchor=base west] at (0.24855775,-6.25) {$\div$};
\node[anchor=base west] at (3.1685576,-6.21) {$\div$};

\node[anchor=base west] at (4.4,-3.93) {$\deff$};
\node[anchor=base west] at (10.0,-3.97) {$\rot\rot$};

\draw[colour0,line width=0.04cm,
      dashed,
      dash pattern=on 0.17638889cm off 0.10583334cm]
  (1.7485577,-3.03) -- (1.7285577,-4.71);
\draw[colour0,line width=0.04cm,
      dashed,
      dash pattern=on 0.17638889cm off 0.10583334cm]
  (0.10855774,-5.87) -- (1.7285577,-4.69);
\draw[colour0,line width=0.04cm,
      dashed,
      dash pattern=on 0.17638889cm off 0.10583334cm]
  (3.3685577,-5.79) -- (1.7285577,-4.69);

\draw[colour0,line width=0.04cm,
      dashed,
      dash pattern=on 0.17638889cm off 0.10583334cm]
  (7.448558,-3.25) -- (7.428558,-4.93);
\draw[colour0,line width=0.04cm,
      dashed,
      dash pattern=on 0.17638889cm off 0.10583334cm]
  (5.8085575,-6.09) -- (7.428558,-4.91);
\draw[colour0,line width=0.04cm,
      dashed,
      dash pattern=on 0.17638889cm off 0.10583334cm]
  (9.068558,-6.01) -- (7.428558,-4.91);

\draw[colour0,line width=0.04cm,
      dashed,
      dash pattern=on 0.17638889cm off 0.10583334cm]
  (13.628558,-3.29) -- (13.608558,-4.97);
\draw[colour0,line width=0.04cm,
      dashed,
      dash pattern=on 0.17638889cm off 0.10583334cm]
  (11.988558,-6.13) -- (13.608558,-4.95);
\draw[colour0,line width=0.04cm,
      dashed,
      dash pattern=on 0.17638889cm off 0.10583334cm]
  (15.248558,-6.05) -- (13.608558,-4.95);

\foreach \x/\y in {
  13.618558/-3.21,
  12.018558/-6.05,
  15.278558/-6.05,
  13.598557/-4.93,
  13.638557/-4.13,
  12.898558/-5.45,
  14.378558/-5.49,
  12.818558/-4.55,
  14.538558/-4.67,
  13.638557/-6.05
}{
  \fill[lightgray] (\x,\y) rectangle ++(0.16cm,0.16cm);
}

\node[anchor=base west] at (7.408558,-3.01) {$\tr$};
\node[anchor=base west] at (5.7685575,-6.43) {$\tr$};
\node[anchor=base west] at (8.828558,-6.37) {$\tr$};
\node[anchor=base west] at (7.5685577,-4.65) {$\tr$};

\draw[line width=0.04cm] (6.7085576,-3.87) -- (7.008558,-3.39);
\draw[line width=0.04cm] (6.1885576,-4.91) -- (5.848558,-5.47);
\draw[line width=0.04cm] (5.848558,-5.47) -- (5.828558,-5.47);
\draw[line width=0.04cm] (7.868558,-3.61) -- (8.168558,-4.05);
\draw[line width=0.04cm] (8.868558,-5.21) -- (8.548557,-4.69);
\draw[line width=0.04cm] (6.7085576,-6.23) -- (7.2085576,-6.23);
\draw[line width=0.04cm] (7.7085576,-6.23) -- (8.288558,-6.23);
\draw[line width=0.04cm] (7.328558,-3.83) -- (7.3085575,-4.23);
\draw[line width=0.04cm] (7.3085575,-4.41) -- (7.3085575,-4.81);
\draw[line width=0.04cm] (6.368558,-5.49) -- (6.6885576,-5.23);
\draw[line width=0.04cm] (6.888558,-5.17) -- (7.1885576,-4.93);
\draw[line width=0.04cm] (7.5885577,-5.15) -- (7.908558,-5.37);
\draw[line width=0.04cm] (8.068558,-5.49) -- (8.508557,-5.79);

\end{tikzpicture}
  \caption{trace enhanced Regge complex \eqref{eq:dicsrete-2D-low} on CT split}
  \label{fig:complex2dhct}
\end{figure}
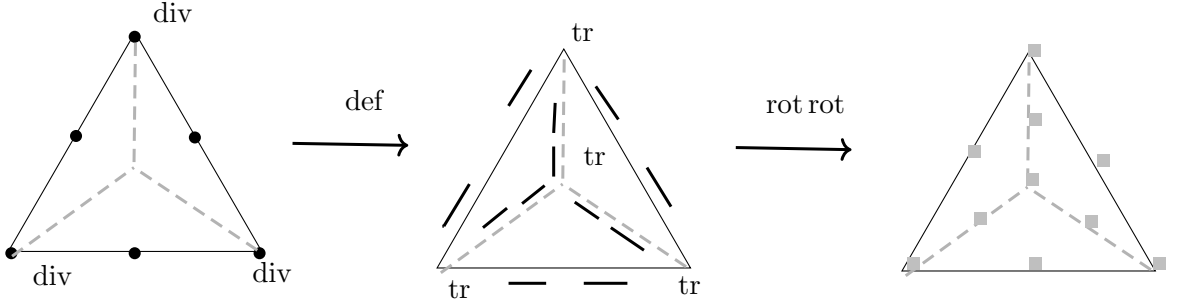

We now compute the characteristic of the complex:
\begin{align}
  & \dim \Upsilon^{CT} - \dim  \overline{\Sigma}(\calT^{CT}) +\dim \lag_{2}^{\ast}  \\
   & = 2|\calV| + |\calV^{\circ}| + |\calE^{\circ}|  - 3|\calF|, \\ 
&	 = 3 \chi - |\calV^{\partial}| + 3|\calE| +  |\calE^{\circ}| - 6|\calF|, \\
	& = 3\chi - |\calV^{\partial}| + |\calE^\partial| + 2|\calE| +  2 |\calE^{\circ}| - 6|\calF|, \\ 
& =  3\chi.
\end{align}

More precisely, we have the following cohomological result. 
\begin{theorem}
  \label{prop:discrete-2D-low}
  The cohomology of \eqref{eq:dicsrete-2D-low} is $\CdR^\bs(\Omega) \otimes \RM_2$. 
\end{theorem}

We now introduce several variants of trace enhanced Regge complexes on HCT splits. We try to reduce as much as possible the number of degrees of freedom. First, recall that $\Upsilon^{CT}(\calT)$ has $\rmC^1$ continuity (supersmoothness) at the inpoint, so that for $u \in \Upsilon^{CT}$, $\deff u$ is continuous there. This indicates that we could impose the full continuity in $\overline{\Sigma}(\calT^{CT})$ at the inpoints $x_{T}$. Explicitely we define:
\begin{equation}
  \widehat{\Sigma}(\calT)= \{\sigma \in \overline{\Sigma}(\calT^{CT}) \ : \ \sigma \in \rmC^0(x_T)\otimes \bbS_2\}.
\end{equation}
As a result, in the distributional space we will be able to discard the deltas at inpoints, which gives the following complex:
\begin{equation}
\label{eq:supersmoothess-ct}
\begin{tikzcd}
0 \ar[r] &  \Upsilon^{CT}(\calT) \ar[r,"\deff"] &  \widehat{\Sigma}(\calT)  \ar[r,"\rot\rot"]  & \lag_2^\ast (\calT) + \bigoplus_{e \in \calE^{CT}} \bbR \delta_e  \ar[r] &  0.
\end{tikzcd}
\end{equation}

Next we can introduce the following space to reduce further the degrees of freedom inside each $T^{CT}$ as much as possible. We define $\widehat{\widehat{\Sigma}}(\calT)$ as follows. The shape function space is defined as: 
\begin{equation}
  \widehat{\widehat{\Sigma}}(T)= \{ \sigma \in \widehat{\Sigma}(T) \ : \ \rot\rot\sigma = 0 \text{ in the interior of } T\}.
\end{equation}
The degrees of freedom are:
\begin{itemize}
    \item $\tr \sigma (x)$ for $x \in \calV(T)$,
    \item $\int_e p \sigma_{\tang\tang}$ for $p \in \rmP_1(e)$ and $e \in \calE(T)$,
    \item $\int_T \sigma$ (in $\bbS_2$). 
\end{itemize}
One can check that the above degrees of freedom are unisolvent. The following is a complex: 
\begin{equation}
\label{eq:reduced-ct}
\begin{tikzcd}
0 \ar[r] &  \Upsilon^{CT}(\calT) \ar[r,"\deff"] &  \widehat{\widehat{\Sigma}}(\calT)  \ar[r,"\rot\rot"]  & \lag_2^\ast (\calT)  \ar[r] &  0.
\end{tikzcd}
\end{equation}

Unfortunately $\widehat{\widehat{\Sigma}} (T)$ does not contain $\lambda_a \bbI$, where $\lambda_a$ is the barycentric coordinate of the central vertex in the CT split, denoted $a$ here. So we can add this matrix field to obtain the following space:
\begin{equation}
  \widetilde{\Sigma}(T) = \widehat{\widehat{\Sigma}}(T) \oplus \bbR \lambda_{a} \bbI.
\end{equation}
Then we add $\Delta \lambda_a$ to the last space in the complex, namely $\lag_2^\ast(\calT)$. The extra degree of freedom in $\widetilde \Sigma(T)$ is:
\begin{itemize}
\item $ \int_T \lambda_a \rot\rot \sigma$.
  \end{itemize}
Correspondingly the extra degree of freedom in the last space is integration against $\lambda_a$. See \Cref{fig:reduced-ct} for an illustration. 
\begin{figure}[htbp]
 
  \begin{tikzpicture}[x=1cm,y=1cm,yscale=1,isosceles triangle stretches= true]

\definecolor{colour0}{rgb}{0.7019608,0.7019608,0.7019608}

\node[draw,isosceles triangle,
      minimum width=3.3692393cm,
      minimum height=2.9132936cm,anchor = lower side, shape border rotate = 90] at (1.7431774,-5.8301883) {};
\node[draw,isosceles triangle,
      minimum width=3.3692393cm,
      minimum height=2.9132936cm,anchor = lower side, shape border rotate = 90] at (7.443177,-6.050188) {};
\node[draw,isosceles triangle,
      minimum width=3.3692393cm,
      minimum height=2.9132936cm,anchor = lower side, shape border rotate = 90] at (13.623178,-6.090188) {};

\draw[->,line width=0.04cm] (3.8368013,-4.385271) -- (5.3758736,-4.4114294);
\draw[->,line width=0.04cm] (9.729523,-4.4408803) -- (11.268558,-4.4692335);

\fill (1.7385577,-2.9668946) circle[radius=0.08cm];
\fill (0.09855774,-5.8468947) circle[radius=0.08cm];
\fill (3.3985577,-5.8468947) circle[radius=0.08cm];
\fill (0.9585577,-4.2868943) circle[radius=0.08cm];
\fill (2.5385578,-4.3068943) circle[radius=0.08cm];
\fill (1.7385577,-5.8468947) circle[radius=0.08cm];

\node[anchor=base west] at (1.8485577,-2.7668946) {$\div$};
\node[anchor=base west] at (0.24855775,-6.2668943) {$\div$};
\node[anchor=base west] at (3.1685576,-6.2268944) {$\div$};

\node[anchor=base west] at (4.1285577,-3.9468946) {$\rot$};
\node[anchor=base west] at (10.368558,-3.9868946) {$\div$};

\draw[colour0,line width=0.04cm,
      dashed,
      dash pattern=on 0.17638889cm off 0.10583334cm]
  (1.7485577,-3.0468946) -- (1.7285577,-4.7268944);
\draw[colour0,line width=0.04cm,
      dashed,
      dash pattern=on 0.17638889cm off 0.10583334cm]
  (0.10855774,-5.8868947) -- (1.7285577,-4.7068944);
\draw[colour0,line width=0.04cm,
      dashed,
      dash pattern=on 0.17638889cm off 0.10583334cm]
  (3.3685577,-5.8068943) -- (1.7285577,-4.7068944);

\draw[colour0,line width=0.04cm,
      dashed,
      dash pattern=on 0.17638889cm off 0.10583334cm]
  (7.448558,-3.2668946) -- (7.428558,-4.9468946);
\draw[colour0,line width=0.04cm,
      dashed,
      dash pattern=on 0.17638889cm off 0.10583334cm]
  (5.8085575,-6.1068945) -- (7.428558,-4.9268947);
\draw[colour0,line width=0.04cm,
      dashed,
      dash pattern=on 0.17638889cm off 0.10583334cm]
  (9.068558,-6.0268946) -- (7.428558,-4.9268947);

\fill[lightgray] (13.618558,-3.2268946) rectangle ++(0.16cm,0.16cm);

\draw[colour0,line width=0.04cm,
      dashed,
      dash pattern=on 0.17638889cm off 0.10583334cm]
  (13.628558,-3.3068945) -- (13.608558,-4.9868946);
\draw[colour0,line width=0.04cm,
      dashed,
      dash pattern=on 0.17638889cm off 0.10583334cm]
  (11.988558,-6.1468945) -- (13.608558,-4.9668946);
\draw[colour0,line width=0.04cm,
      dashed,
      dash pattern=on 0.17638889cm off 0.10583334cm]
  (15.248558,-6.0668945) -- (13.608558,-4.9668946);

\node[anchor=base west] at (7.408558,-3.0268946) {$\tr$};
\node[anchor=base west] at (5.7685575,-6.4468946) {$\tr$};
\node[anchor=base west] at (8.828558,-6.3868947) {$\tr$};
\node[anchor=base west] at (7.4385576,-4.9468946) {$\rot \rot$};

\draw[line width=0.04cm] (6.7085576,-3.8868945) -- (7.008558,-3.4068944);
\draw[line width=0.04cm] (6.1885576,-4.9268947) -- (5.848558,-5.4868946);
\draw[line width=0.04cm] (5.848558,-5.4868946) -- (5.828558,-5.4868946);
\draw[line width=0.04cm] (7.868558,-3.6268945) -- (8.168558,-4.0668945);
\draw[line width=0.04cm] (8.868558,-5.2268944) -- (8.548557,-4.7068944);
\draw[line width=0.04cm] (6.7085576,-6.2468944) -- (7.2085576,-6.2468944);
\draw[line width=0.04cm] (7.7085576,-6.2468944) -- (8.288558,-6.2468944);

\foreach \x/\y in {
  12.018558/-6.0668945,
  15.278558/-6.0668945,
  12.818558/-4.5668945,
  14.538558/-4.6868944,
  13.638557/-6.0668945,
  13.638/-4.96
}{
  \fill[lightgray] (\x,\y) rectangle ++(0.16cm,0.16cm);
}

\fill (7.4385576,-4.9468946) circle[radius=0.08cm];

\end{tikzpicture}
  
  \caption{Reduced trace enhanced Regge complex based on $\widetilde{\Sigma}(T)$ on CT split.}
  \label{fig:reduced-ct} 
\end{figure}
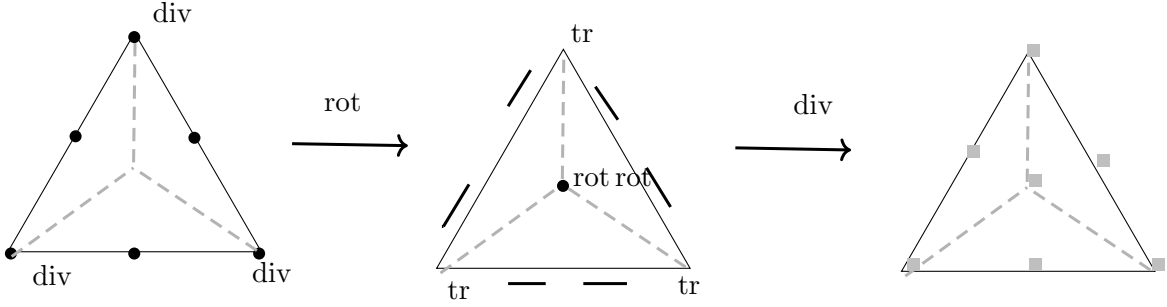


\section{Trace enhanced Regge elements in three dimensions\label{sec:dimthree}}
In this section, we consider the space
\begin{equation}
  \Htr(\Omega) = \{\sigma \in  \Hinc(\Omega) \ : \ \tr \sigma \in \rmH^1(\Omega)\},
\end{equation}
in three dimensions. We remark first:

\begin{proposition}
We have $\Htr(\Omega) \subseteq  \Hdivdiv(\Omega)$.
\end{proposition}
\begin{proof}
  In three dimensions we have the identity:
  \begin{equation}
    \tr \inc \sigma + \div \div \sigma = \Delta \tr \sigma,
  \end{equation}
  which lets us conclude.
\end{proof}

The trace enhanced Sobolev space can be fitted into the following elasticity complex in three dimension. 
\begin{equation}
\label{eq:Sobolev-3D}
\begin{tikzcd}
0 \ar[r] & \Hdiv^1(\Omega) \ar[r,"\deff"] &  \Htr(\Omega) \ar[r,"\inc"] & \Hsymdiv^{-1}(\Omega) \ar[r,"\div"] & \urmH^{-1}(\Omega) \ar[r] & 0.
\end{tikzcd}
\end{equation}

It follows from Lemma \ref{lem:regcoh} as in Proposition \ref{prop:2dinccoh}, that the inclusion of this complex into the following more standard complex induces isomorphisms on cohomology:
\begin{equation}
\label{eq:Sobolev-3D-std}
\begin{tikzcd}
0 \ar[r] & \urmH^1(\Omega) \ar[r,"\deff"] &  \Hinc(\Omega) \ar[r,"\inc"] & \Hsymdiv^{-1}(\Omega) \ar[r,"\div"] & \urmH^{-1}(\Omega) \ar[r] & 0.
\end{tikzcd}
\end{equation}
The cohomology of the latter is isomorphic to $\CdR(\Omega) \otimes \RM_3$ by BGG techniques \cite{ArnHu21}.

For the finite element construction, we require that the shape functions should be tangential-tangential continuous and normal-normal continuous on faces. Again we obtain a non-conforming method (the $\inc$ is almost $\rmH^{-1}$ conforming but the $\div$ is certainly not $\rmH^{-1}$ conforming). 

\subsection{Construction on general meshes}

We begin with the following lemma.

\begin{lemma}	
\label{lem:reggebubble}
Define $\uuB^{\text{reg}}(T)$ as the Regge bubble function space:
\begin{equation}
  \uuB^{\text{reg}}(T) = \{ \sigma \in \rmP_k(T) \otimes \bbS_3 \ : \ \Pi_f \sigma \Pi_f = 0 \text{ for } f \in \calF(\partial T) \}.
\end{equation}
Then, it holds that $\uuB^{\text{reg}}(T)$ is a direct sum of the six spaces\footnote{Unfortunately $k$ has two meanings here. We hope this does not create confusion.}:
\begin{equation}
  \lambda_{i}\lambda_{j} \rmP_{k-2}(T) (\norm_k \odot \norm_l),
\end{equation}
relative to edges $[i,j]$ with opposite edge $[k,l]$. 
As a consequence, the dimension of $ \uuB^{\text{reg}}(T)$ is $6 \binom{k+1}{3}$. 
\end{lemma}

\begin{proof}
It is a three-dimension specification of results in \cite[\S 2.2.2]{Li18}.
\end{proof}
By the above characterisation, tensors in $\uuB(T)^{\text{reg}}$ vanish at each vertex in $\calV(T)$. 

We now define the Regge bubble with trace modification. Let 
\begin{equation}
  \uuB(T) = \{ \sigma \in \uuB(T)^{\text{reg}} \ : \ \tr \sigma = 0 \text{ on } \partial T\},
  \end{equation}
and:
\begin{equation}
  \widetilde{\uuB}(T) = \{ \sigma \in \uuB(T) \ : \ \sigma|_{e} = 0 \text{ for all edges } e \text{ of } \partial T\}.
\end{equation}
\begin{remark}
If $\norm_k \cdot \norm_l \neq 0$ (dihedral angle different from $\pi/2$) for all edges $[k,l]$, it holds that $\uuB(T) = \widetilde{\uuB}(T)$. 
\end{remark}

We then have the following result.
\begin{proposition}
  The dimension of $\uuB(T)$ is $k^3 - 2 k^2 - k + 2$. 
\end{proposition}
\begin{proof}
Suppose that $\sigma \in \uuB(T)$. By \Cref{lem:reggebubble}, there exist uniquely determined $p_{ij} \in \rmP_{k-2}(T)$, such that 
\begin{equation}
  \sigma = \sum_{i \neq j} \lambda_i \lambda_j p_{ij} (\norm_k \odot \norm_l).
\end{equation}
Then we have: 
\begin{equation}
  \tr \sigma = \sum_{i \neq j} \lambda_i \lambda_j p_{ij} (\norm_k \cdot \norm_l).
\end{equation}
On each of the $6$ edges $e_{ij}$, it holds that $p_{ij}|_{e_{ij}} = 0$, which gives $(k-1)$ constraints.\\
On each of the $4$ faces $f$ we have $\binom{k-1}{2}$ constraints (taking into account that the degrees have been lowered by one by the edge constraint)\shc{independence?}.\\
Therefore, the total dimension of $\uuB(T)$ is 
\begin{equation}
  6\binom{k+1}{3} - 4\binom{k-1}{2} - 6(k-1) = k^3 - 2 k^2 - k + 2.
\end{equation}
This was the claim.
\end{proof}

We can now define the finite element $\Sigma_{k}$ in three dimensions. The underlying function space is $\rmP_k(T) \otimes \bbS_3$. For $\sigma \in \rmP_k(T) \otimes \bbS_3$, we consider the following degrees of freedom:
\begin{itemize}
\item  $\sigma(x)$ for all vertex $x \in \calV(T)$ ($6$ per vertex)
\item  $\int_e \sigma p$ for $p \in \rmP_{k-2}(e) \otimes \bbS_3$ for $e \in \calE(T)$ ($6(k-1)$ per edge)
\item  $\int_{f} \sigma_{\tau\tau} :  b$ for $b \in \rmP_{k-3}(f) \otimes \bbS(f)$ for each $f \in \calF(T)$ ($3 \binom{k-1}{2}$ per face)
\item  $\int_f (\tr \sigma) p $ for $p \in \rmP_{k-3}(f)$ for $f \in \calF(T)$ ($\binom{k-1}{2}$ per face)
\item  $\int_T \sigma : b $ for $b \in \widetilde{\uuB}(T)$ (given above)
\end{itemize}

\begin{proposition}
  The proposed degrees of freedom for $\Sigma_k$ are unisolvent. 
\end{proposition}
\begin{proof}
It suffices to count dimensions. The number of degrees of freedom in each of the above sets is: 
\begin{itemize}
\item  $24$
\item  $36(k-1)$
\item  $6(k-1)(k-2)$ \shc{changed}
\item  $2(k-1)(k-2)$ \shc{changed}
\item  $k^3 - 2 k^2 - k + 2$
\end{itemize}
The total number is computed to be:
\begin{equation}
 6\binom{k+3}{3} = \dim \rmP_k(T) \otimes \bbS_3.
\end{equation}
This concludes the proof.
\end{proof}
The global space is continuous at vertices and edges. It has continuous tangential-tangential and normal-normal components on faces. In particular the resulting space is $ \Htr(\Omega)$ conforming.

\subsection{Trace enhanced element space on Worsey-Farin split}

In this subsection, we construct a trace enhanced Regge element space on the Worsey-Farin split with piecewise linear tensors.

We first briefly recall the definition of the Worsey-Farin split. For each tetrahedron $T$, we refine it into 12 small tetrahedra based on a choice of inpoint $x_T$ of the cell and $x_f$ of each face $x_f$.  For two cells $T_1$ and $T_2$ sharing a face $f$ it is required that $x_f$ lies on the segment $[x_{T_1}, x_{T_2}]$. We denote this refinement as $T^{WF}$. The refined mesh will be denoted as $\mathcal{T}^{WF}$, and the sets of vertices, edges, and faces are denoted as $\mathcal{V}^{WF}$, $\mathcal{E}^{WF}$, and $\mathcal{F}^{WF}$, respectively. Note that for each face $f \in \mathcal{F}$, the Worsey-Farin split yields a Clough-Tocher split $f^{CT}$, and the notation $\mathcal{E}^{CT}(f)$ will be used.

We recall the following $\Hdiv^1$-conforming element from \cite{GuzLisNei22} (in \cite[Theorem 4]{ChrHu18} a corresponding linear element on the WF split with divergence on the Alfeld split is described). Here the local shape space is: 
\begin{equation}
\Upsilon^{WF}(T) =  \{ u \in \lag_2(T^{WF}) \otimes \bbR^3 \ : \ \div u \in \lag_1(T^{WF}) \}.
\end{equation}
The degrees of freedom are given by \cite[Lemma 5.13]{GuzLisNei22}, specified here in the quadratic case:
\begin{itemize}
	\item  $ u(x)$,  $\div u(x)$ for each vertex $x \in \calV(T)$ (4 for each vertex)
	\item  $\int_e u$ for each edge $e \in \calE(T)$ (3 for each edge)
	\item  $\int_{f} (u \cdot \norm_f) p $ for  $p \in \lag_{1}(f^{CT}) $ (4 for each face)
	\item  $\int_{f} \Pi_f  u \cdot p $ for $p \in \lag_2(f^{CT}) \otimes f$ such that  $\div p \in \lag_1(f^{CT})$ and $p|_{\partial f} = 0$ (3 for each face)
	\item  $\int_{f} \div u$   (1 for each face)
\end{itemize}
The local dimension is 66. The global dimension is $4|\calV| + 3|\calE| + 8|\calF|$.


Now we construct the trace enhanced Regge element on the Worsey-Farin split. We let:
\begin{equation}
\Sigma^{WF}(T) =  \{\sigma \in \reg_1(T^{WF}) \ : \  \tr \sigma  \in \lag_1(T^{WF})\}.
\end{equation}

\begin{lemma}
The dimension of $\Sigma^{WF}(T)$ is 103.	
\end{lemma}
\begin{proof}
  By \cite{GuzLisNei22}, working in $T^{WF}$, we have surjectivity of:
  \begin{equation}
    \div : \lag_2(T^{WF})\otimes \bbR^3 \to \rmD \rmP_1 (T^{WF}),
  \end{equation}
  mapping into discontinuous $\rmP_1$. Therefore it follows that $\tr \reg_1(T^{WF}) = \rmD \rmP_1 (T^{WF})$. This leads to:
  \begin{align}
    \dim \Sigma^{WF}(T) & = \dim \reg_1(T^{WF}) - \div \rmD \rmP_1(T^{WF}) + \div \lag_1(T^{WF}) ,\\
    & = 142 - 48 +9 = 103.
  \end{align}
  This was the claim.
\end{proof}

The degrees of freedom of $\Sigma^{WF}(T)$ are given as follows.
\begin{itemize}
	\item  $\tr \sigma(x) $ for $  x \in \calV(T)$ (count 1 each vertex)
	\item  $\int_e \sigma_{\tang\tang} p $ for $ p \in \rmP_1(e),  e \in \calE(T)$ (count 2 for each edge)
	\item  $\int_{e'} \sigma_{\tang\tang} p $ for $  p \in \rmP_1(e'), \forall e' \in \calE^{\circ}(f^{CT}), f \in \calF(T)$ (count 6 for each face)
	\item  $\int_{f} (\tr_f \Pi_f\sigma\Pi_f) q$ for $  q \in \lag_1(f^{CT}), f \in \calF(T)$ (count 4 for each face)
	\item  $\int_f \sigma_{\norm \norm} $ for $ f \in \calF(T)$ (count 1 for each face)
	\item  $\int_T \sigma : b$ for $b \in \mathring{\reg}_1(T^{WF}) \cap (\deff \mathring{\lag}_2(T^{WF}))^{\perp}$
\end{itemize}

Here, $\mathring{\reg_1}(T^{WF})$ is the space of piecewise linear Regge elements with zero tangential-tangential boundary conditions on $\partial T$. The dimension of this space is 70. On the other hand, $\mathring{\lag}_2(T^{WF})$ is the space of piecewise quadratic Lagrange elements with zero boundary conditions on $\partial T$, and its dimension is 27. It follows that $\dim \mathring{\Sigma}^{WF}(T)=43$.
        
We now show unisolvence. 
\begin{proposition}
	The proposed degrees of freedom for $\Sigma^{WF}$ are unisolvent. The resulting space is $\Htr$-conforming. The global dimension is $|\calV| + 2|\calE| + 11|\calF| + 43|\calT|$. 
\end{proposition}

\begin{proof}
Since $1\times4 + 2\times 6+11\times 6 + 43=103$, it suffices to show that if $\sigma \in \Sigma^{WF}(T)$ has vanishing degrees of freedom, then $\sigma = 0$.

We first prove that for each face $f$, $\sigma_f=\Pi_f\sigma\Pi_f$ vanishes. Fixing $f$, we have $\int_e(\sigma_f)_{\tang\tang}p = 0$ for all $p\in \rmP_1(e)$ and for every $e\in\mathcal{E}(f^{CT})$, and $\int_f\mathrm{tr}_f\sigma_f\cdot q = 0$ for all $q\in\lag_1(f^{CT})$. Then, by the previously obtained two-dimensional result, viz. \Cref{prop:2d-low}, we can conclude that $\sigma_f = 0$.

Thus, $\sigma\in\mathring{\reg}_1(T^{WF})$. From the last set of degrees of freedom, we obtain $\sigma\in \deff \mathring{\lag}_2(T^{WF})$. Write $\sigma =\deff u$ for some $u\in\mathring{\lag}_2(T^{WF})$. Since $\sigma \in\Sigma^{WF}(T)$, we have $\div u \in\lag_1(T^{WF})$. Therefore, $u \in \Upsilon^{WF}(T)$ and has vanishing degrees of freedom w.r.t. $\Upsilon^{WF}(T)$ (using here the second to last degrees of freedom), which implies that $u = 0$. This completes the proof.

\end{proof}

Using these function spaces, we can propose the trace enhanced Regge complex in three dimensions. 
%
%
%
%
%
\begin{theorem}
The following spaces and operators form a complex:
\begin{equation}
  \label{eq:dicsrete-3d-WF}
  \begin{tikzcd}
    0 \ar[r] &  \Upsilon^{WF}(\calT) \ar[r,"\deff"] &  \Sigma^{WF}(\calT)  \ar[r,"\inc"]  & \mathring{\reg}^{\ast}_1(\calT^{WF}) \ar[r,"\div"]  & \mathring{\lag}^{\ast}_2(\calT^{WF}) \ar[r] &  0
  \end{tikzcd}
\end{equation}
\end{theorem}

\begin{proposition}
  The inclusion of \eqref{eq:dicsrete-3d-WF} in the linear Regge complex \eqref{eq:regge-3d-k} (relative to $k=1$ and the WF refined mesh) induces isomorphisms on cohomology.
\end{proposition}
\begin{proof}
  We use Lemma \ref{lem:regcoh}.

  Suppose $\sigma \in \reg_1(\calT^{WF}$. Then there is $u \in \lag_2(\calT^{WF}) \otimes \bbR^3$ such that $\div u = \tr \sigma$. Then $\sigma - \deff u \in \Sigma^{WF}(\calT)$.

  Suppose $u \in \lag_2(\calT^{WF})$ and $\deff u \in \Sigma^{WF}(\calT)$. Then $\div u$ is continuous so $u \in  \Upsilon^{WF}(\calT)$.

  This is all we need to check.
\end{proof}


\appendix

\section{Regarding cohomologies}
\label{sec:cohomology}

In this section, we present the remaining proofs about cohomologies. The techniques extend those of \cite{HuLinZha25,ChrHuLin26} and illustrate the use of spectral sequences in a finite element context.

\subsection{Cohomology of high order Regge complexes in dimension two\label{sec:cohhess2d}}
We first show the proof concerning the high order Regge complex (cf. \Cref{thm:high-regge}) :
\begin{equation}
\label{eq:2Ddivdiv-Lag}
\begin{tikzcd}
0 \ar[r] & \lag_{k+1}(\mathcal T) \otimes \mathbb R^2 \ar[r,"\deff"] &  \reg_{k}(\mathcal T) \ar[r,"\rot\rot"] &  \mathring{\lag}_{k+1}^{\ast}(\mathcal T) \ar[r] & 0.
\end{tikzcd}
\end{equation}

The strategy of the proof is sketched as follows and is similar to that of \cite{HuLinZha25}.  We first show the result for the adjoint complex (notice boundary conditions):
\begin{equation}
\label{eq:2Dstress-Lag}
\begin{tikzcd}
0 \ar[r] & \lag_{k+1}(\mathcal T) \ar[r,"\rot \rot "] &  \mathring{\reg}_{k}^\ast(\mathcal T) \ar[r,"\div"] &  \mathring{\lag}_{k+1}^{\ast}(\mathcal T) \otimes \bbR^2 \ar[r] & 0.
\end{tikzcd}
\end{equation}
Here $\rot \rot = \airy$ maps scalars to stress tensors. Recall that:
\begin{equation}
  \mathring{\reg}_{k}^\ast(\mathcal T)  = \bigoplus_{e \in \mathcal E^{\circ}} \rmP_{k}(e) \tang \tang^\transp \delta_e \oplus \bigoplus_{f \in \mathcal F} \rmP_{k-1}(f)\bbS \delta_f,
\end{equation}
is the dual space of $\mathring{\reg}_k$.

Through standard algebraic isomorphisms that are dual to those in \eqref{eq:straindivdiv}, the Airy complex is isomorphic to the Hessian complex, which is based on $\hess = \grad \grad$:
\begin{equation}
\label{eq:2DHessian-Lag}
\begin{tikzcd}
0 \ar[r] & \lag_{k+1}(\mathcal T) \ar[r,"\hess"] & \calS \mathring{\reg}_{k}^\ast(\mathcal T) \ar[r,"\rot"] &  \mathring{\lag}_{k+1}^{\ast}(\mathcal T) \otimes \bbR^2 \ar[r] & 0.
\end{tikzcd}
\end{equation}

\begin{proposition}
  For any $k\geq 1$, the cohomology of \eqref{eq:2DHessian-Lag} is isomorphic to $\CdR^\bs(\Omega) \otimes \rmP_1$. 
\end{proposition}

\begin{proof}
The case when $k = 1$ is treated in \cite{HuLinZha25}, and is an important tool in this proof. 

First, we consider the following double complex:
\begin{equation}
\label{diag:hessian}
\begin{tikzcd}
	\lag_{k+1} \ar[r,"\hess_h"] \ar[d,"\text{airy}_d"] & \reg_{k-1} \ar[r,"{\rot_h}"] \ar[d,"{\div_d}"] & \rmD \rmP_{k-2}(\calT) \bbR^2  \\ 
	\bigoplus_{e \in \mathcal E^{\circ} } \rmP_{k}(e) \norm \norm^\transp \delta_{e} \ar[r,"{\rot_h}"]  \ar[d,"\div_d"]  & \bigoplus_{e \in \calE^{\circ} }  \rmP_{k-1}(e) \norm \delta_e  \\ 
	\bigoplus_{x \in \calV^{\circ} }  \bbR^2 \delta_x
\end{tikzcd}
\end{equation}
Here, the horizontal operators extract the piecewise defined differential operators. Notice in particular that $\hess_h$ applied to $\lag_{k+1}$ preserves tangential-tangential continuity. The vertical operators collect remaining distributional parts of the operators, consisting of jumps on lower dimensional cells (boundary terms in integrations by parts).

Now we compute the cohomology of its total complex by a spectral sequence approach. We first compute the horizontal cohomology:
\begin{equation}
\label{diag:hessian-ss1}
\begin{tikzcd}
	\lag_1 \ar[r] \ar[d] & {0} \ar[r] \ar[d] & {0} \\ 
	 \bigoplus_{e \in \mathcal E^{\circ}} \rmP_0(e) \norm \norm^\transp \delta_{e} \ar[r]  \ar[d]  & 0 \\ 
	\bigoplus_{x \in \calV^{\circ} }  \bbR^2 \delta_x.
\end{tikzcd}
\end{equation}
We postpone the computation of the two last cohomology groups on the top row.

The first column reduces to the first order case, which is treated in \cite{HuLinZha25}. Taking the vertical cohomology, we obtain 
\begin{equation}
\begin{tikzcd}
	 \CdR^0(\Omega) \otimes \rmP_1  \ar[r] \ar[d] & 0 \ar[r] \ar[d] & 0 \\ 
	 \CdR^1(\Omega) \otimes \rmP_1   \ar[r]  \ar[d]  &  0  \\ 
	 \CdR^2(\Omega) \otimes \rmP_1.
\end{tikzcd}
\end{equation}
The spectral sequence then converges. We now consider the total complex of the double complex \eqref{diag:hessian}:
\begin{equation}
\label{eq:2DHessian-Lag-01}
\begin{tikzcd}
\lag_{k+1} \ar[r,"\hess"] &  \reg_{k-1} \oplus \bigoplus_{e \in \calE^{\circ}} \rmP_k(e)\norm \norm^\transp \delta_e  \ar[r,"\rot"] &  Q_{k+1},
\end{tikzcd}
\end{equation}
with:
\begin{equation}
  Q_{k+1}= \rmD \rmP_{k-2}(\calT) \bbR^2 \oplus \bigoplus_{e \in \calE^{\circ} }  \rmP_{k-1}(e) \norm \delta_e \oplus 	\bigoplus_{x \in \calV^{\circ} }  \bbR^2 \delta_x.
\end{equation}

By the preceding arguments, its cohomology is $\CdR^\bs(\Omega) \otimes \rmP_1$.

We now compare \eqref{eq:2DHessian-Lag-01} and \eqref{eq:2DHessian-Lag}. In the tensor-valued distribution at index 1 in \eqref{eq:2DHessian-Lag-01} imposes tangential-tangential continuity on faces, while the vector-valued distribution at index 2 in \eqref{eq:2DHessian-Lag-01} removes the tangential edge deltas. Accordingly we can arrange the spaces in the following diagram:
\begin{equation}
  \begin{tikzcd}
\lag_{k+1} \ar[r,"\hess"] \ar[d, "\id"] &  \reg_{k-1} \oplus \bigoplus_{e \in \calE^{\circ}} \rmP_k(e)\norm \norm^\transp \delta_e  \ar[r,"\rot"] \ar[d] &  Q_{k+1} \ar[d],\\    
\lag_{k+1} \ar[r,"\hess"] \ar[d,""] &  \rmD \reg_{k-1} \oplus \bigoplus_{e \in \mathcal E^{\circ}} \rmP_k(e) \norm \norm^\transp \delta_e  \ar[r,"\rot"]  \ar[d,"\phi^1"] & \mathring{\lag}^\ast_{k+1} \otimes \bbR^2 \ar[d,"\phi^2"] \\ 
0 \ar[r,""] & \bigoplus_{e \in \mathcal E^{\circ}} \rmP_{k-1}(e) \ar[r,"\mathrm{id}"]  &\bigoplus_{e \in \mathcal E^{\circ}} \rmP_{k-1}(e).
\end{tikzcd}
\end{equation}	
The vertical maps are defined as follows:
\begin{align}
  \phi^1 & : \sigma \mapsto \jump{\tang^\transp \sigma \tang}_e\\
  \phi^2 & : u \mapsto u_e \cdot \tang_e.
\end{align}
The vertical sequences (extended by $0$) are exact and the bottom horizontal sequence is also exact. Therefore, the cohomology of \eqref{eq:2DHessian-Lag-01} is isomorphic to that of \eqref{eq:2DHessian-Lag}. 

We now get back to the remaining terms \eqref{diag:hessian-ss1}. We want to show that:
\begin{equation}
\label{diag:hessian-pw}
\begin{tikzcd}
0 \ar[r] & \lag_{k+1} \ar[r,"\hess_h"] & \reg_{k-1} \ar[r,"\rot_h"] & \rmD\rmP_{k-2} \otimes \mathbb R^2, 
\end{tikzcd}
\end{equation}
has the cohomology $\lag_1 \to 0 \to 0$. 

Consider another double complex:
\begin{equation}
\label{diag:hessian-dg}
\begin{tikzcd}
	\bigoplus_{f \in \calF} \rmP_{k+1}(f) \ar[r,"\hess_h"] \ar[d,"\partial"] & \bigoplus_{f \in \calF} \rmP_{k-1}(f) \otimes \bbS \ar[r,"{\rot_h}"] \ar[d,"{\Pi_e \partial \Pi_e}"] & \rmD \rmP_{k-2} \otimes \bbR^2  \\ 
	\bigoplus_{e \in \mathcal E^{\circ} } P_{k+1}(e)  \ar[r,"{\partial_{\tang}^2}"]  \ar[d,"\partial"]  & \bigoplus_{e \in \mathcal E^{\circ} }  P_{k-1}(e)  \\ 
	\bigoplus_{v \in \mathcal V^{\circ} }  \bbR.
\end{tikzcd}
\end{equation}
The boundary operators in first column are standard. They collect jumps of functions with relative orientations. That column is a resolution of $\lag_{k+1}$ by arguments in \cite{Lic17}.

The boundary operator $\Pi_e \partial \Pi_e$ in the second column is the jump of the tangential tangential component. That column resolves $\reg_{k-1}$ by the results in \cite{Li18}.

Thus the vertical cohomology gives \eqref{diag:hessian-pw} on the to row.

The horizontal cohomology gives the following first column:
\begin{equation}
\label{eq:P-1-cmplx}
\begin{tikzcd}
  \bigoplus_{f \in \calF} \rmP_1(f) \ar[d,"\partial"]\\
   \bigoplus_{e \in \calE^{\circ} } \rmP_1(e)  \ar[d,"\partial"] \\
  \bigoplus_{v \in \mathcal V^{\circ} }  \bbR.
\end{tikzcd}
\end{equation}
By spectral sequence theory, the cohomology of \eqref{diag:hessian-pw} is isomorphic to the cohomology of \eqref{eq:P-1-cmplx}, which is $\lag_1\to 0 \to 0$, again by  \cite{Lic17}. 
\end{proof}

\subsection{A lemma on subcomplexes}

\begin{lemma}\label{lem:regcoh}
Suppose that $A^\bs$ is a subcomplex of a complex $B^\bs$. Suppose that for each $k$, for any $u \in B^k$ such that $\rmd u \in A^{k+1}$ there exists $v \in B^{k-1}$ such that $u - \rmd v \in A^k$. Then the canonical injection of $A^\bs$ into $B^\bs$ induces isomorphisms on cohomology.    
\end{lemma}

\begin{proof} By elementary means.

Injectivity. Suppose $u \in A^k$ satifies $d u = 0$ and $u = dv$ for some $v \in B^{k-1}$. Then we can write $v = v' + d w$ with $v' \in A^{k-1}$ and $w \in B^{k-2}$. Then $u = d v'$. In other words if the class of $u$ is mapped to the $0$ class in $B^k$, then actually the class of $u$ is zero in $A^k$.

Surjectivity. Suppose that $v \in B^k$ satisfies $d v = 0$. Then write $v = v' + d w$ with $v' \in A^k$ and $w \in B^{k-1}$. Then $d v' = 0$. In other words the class of $v$ in $B^k$ is the image of the class of $v'$ in $A^k$.    
\end{proof}

In practice we often apply this lemma to situations where $A^\bs$ and $B^\bs$ are complexes of Sobolev spaces and $A^\bs$ just encodes higher regularity than $B^\bs$. Then the hypothesis is checked based on regularity theory for PDEs, including such things as regularized Poincaré operators \cite{CosMcI10}. In a discrete setting higher regularity (say of the trace) can also be encoded in the subcomplex.

\section*{Acknowledgements}
As already mentioned, the goal of discretizing the equations of GR was a motivation for developing Finite Element Exterior Calculus \cite{Arn02}. Over the years we have benefited from discussions on this topic with Douglas Arnold, Ragnar Winther and Kaibo Hu.

\bibliographystyle{plain}
\bibliography{../Bibliography/alexandria,../Bibliography/newalexandria,../Bibliography/mybibliography}
\end{document}